\documentclass[12pt]{amsart}
\usepackage{amssymb,verbatim}
\usepackage[active]{srcltx}
\usepackage{graphicx}

\newtheorem{thm}{Theorem}[section]
\newtheorem{lem}[thm]{Lemma}
\newtheorem{cor}[thm]{Corollary}

\newtheorem{prop}[thm]{Proposition}

\theoremstyle{definition}
\newtheorem{dfn}[thm]{Definition}
\newtheorem{example}[thm]{Example}

\newcommand{\ca}{caterpillar}

\newcommand{\R}{\mathbb{R}}
\newcommand{\iU}{U^\infty}
\newcommand{\mz}{$0$-major~}
\newcommand{\mh}{$\frac12$-major~}
\newcommand{\ci}{\mathrm{CI}}
\newcommand{\fg}{\mathrm{FG}}

\newcommand{\car}{\mathrm{CA}}
\newcommand{\ic}{\mathrm{IC}}

\newcommand{\cri}{\mathrm{Cr}}
\newcommand{\orb}{\mathrm{orb}}
\newcommand{\pc}{\mathrm{PC}}

\newcommand{\ro}{\mathrm{Rot_{\frac12}}}
\newcommand{\rot}{\mathrm{Rot_\tau}}

\newcommand{\qml}{\mathrm{QML}}

\newcommand{\cut}{\mathrm{CUT}}
\newcommand{\ce}{\mathrm{COR}}

\newcommand{\smp}{\mathcal{S}}

\newcommand{\F}{\mathcal{F}}

\newcommand{\bj}{\bar j}
\newcommand{\bq}{\bar q}

\newcommand{\di}{\ol{\mathrm{Di}}}
\newcommand{\ol}{\overline}

\newcommand{\0}{\varnothing}
\newcommand{\sm}{\setminus}

\newcommand{\bd}{\mathrm{Bd}}

\newcommand{\ch}{\mathrm{Ch}}

\newcommand{\ph}{\varphi}

\newcommand{\ga}{\gamma}
\newcommand{\si}{\sigma}
\newcommand{\ta}{\theta}
\newcommand{\om}{\omega}

\newcommand{\la}{\lambda}
\newcommand{\nin}{\not\in}

\newcommand{\hell}{\hat{\ell}}

\newcommand{\C}{\mbox{$\mathbb{C}$}}
\newcommand{\hc}{\mbox{$\mathbb{\widehat{C}}$}}
\newcommand{\D}{\mbox{$\mathbb{D}$}}

\newcommand{\bbd}{\mbox{$\mathbb{D}$}}
\newcommand{\disk}{\mathbb{D}}

\newcommand{\ucirc}{\mathbb{S}^1}
\newcommand{\uc}{\mathbb{S}^1}

\newcommand{\tha}{\text{Th}}

\newcommand{\lam}{\mathcal{L}}
\newcommand{\M}{\mathcal{M}}

\def\crA{{\rm CrA}}

\begin{document}

\date{June 24, 2011}
\title[Topological polynomials with a simple core]
{Topological polynomials with a simple core}

\author[A.~Blokh]{Alexander~Blokh}

\thanks{The first named author was partially
supported by NSF grant DMS--0901038}

\author[L.~Oversteegen]{Lex Oversteegen}

\thanks{The second named author was partially  supported
by NSF grant DMS-0906316}

\author[R.~Ptacek]{Ross~Ptacek}

\author[V.~Timorin]{Vladlen~Timorin}

\thanks{The fourth named author was partially supported by
the Deligne fellowship, the Simons-IUM fellowship, RFBR grants 10-01-00739-a, 11-01-00654-a,
MESRF grant MK-2790.2011.1,
RFBR/CNRS project 10-01-93115-NCNIL-a, and
AG Laboratory NRU-HSE, MESRF grant ag. 11 11.G34.31.0023}

\address[Alexander~Blokh, Lex~Oversteegen and Ross~Ptacek]
{Department of Mathematics\\ University of Alabama at Birmingham\\
Birmingham, AL 35294-1170}

\address[Vladlen~Timorin]
{Faculty of Mathematics\\
Laboratory of Algebraic Geometry and its Applications\\
Higher School of Economics\\
Vavilova St. 7, 112312 Moscow, Russia
}

\address[Vladlen~Timorin]
{Independent University of Moscow\\
Bolshoy Vlasyevskiy Pereulok 11, 119002 Moscow, Russia}

\email[Alexander~Blokh]{ablokh@math.uab.edu}
\email[Lex~Oversteegen]{overstee@math.uab.edu}
\email[Ross~Ptacek]{rptacek@uab.edu}
\email[Vladlen~Timorin]{vtimorin@hse.ru}




\begin{abstract}
We define the (dynamical) core of a topological polynomial (and the
associated lamination). This notion extends that of the core of a
unimodal interval map. Two explicit descriptions of the core are
given: one related to periodic objects and one related to critical
objects. We describe all laminations associated with quadratic and
cubic topological polynomials with a simple core (in the quadratic case,
these correspond precisely to points on the Main Cardioid of the
Mandelbrot set).
\end{abstract}

\maketitle

\section{Introduction and the main results}\label{s:intro}


The \emph{complex quadratic family} is the family of all polynomials
$P_c(z)=z^2+c$ (any quadratic polynomial is M\"obius-conjugate to
some $P_c(z)$). A central role in studying this family is played by the
\emph{connectedness locus $\M_2$} (also called the \emph{Mandelbrot
set}) consisting of all $c$ such that the corresponding \emph{Julia
set} $J(P_c)$ is connected.

The core of $\M_2$ is the Principal Hyperbolic Domain ${\rm PHD}_2$
in the parameter space, i.e. the set of parameters whose polynomials
have an attracting fixed point. Its boundary $\car_2$ is called the
\emph{Main Cardioid (of the Mandelbrot set)}. A combinatorial model
for the boundary of $\M_2$, due to Thurston \cite{thu85}, implies  a
combinatorial model of $\car_2$. As it is known that the
combinatorial model of $\car_2$ is homeomorphic to $\car_2$ (and is
homeomorphic to the circle, see e.g. \cite{cg}), we denote the
combinatorial model $\car_2$ as well. Our work is motivated by the
desire to provide a topologically dynamical point of view on these
objects and use it to extend them to the cubic case.

A compactum $Q\subset \C$ is called \emph{unshielded} if it coincides with
the boundary of the unbounded component of $\C\sm Q$. A principal
example of an unshielded compactum is the Julia set of a complex
polynomial map. Unshielded continua have been studied before (see,
e.g., \cite{bo04}), however the notion is useful also when applied to
compacta. A related object is the \emph{topological hull} of a
compactum $Y$, which is defined as the complement to the unbounded component of
$\C\sm Y$ and denoted by $\tha(Y)$ (notice that $\tha(Y)$ is defined
even if $Y$ is not unshielded).

A continuum $X$ in the plane is \emph{tree-like} if $X=\tha(X)$ and $X$ has no interior
points; 
any such continuum is
unshielded. In this case, for every closed set $Y\subset X$, there is a
unique smallest continuum $Z\subset X$ containing $Y$, coinciding
with the intersection of all continua in $X$ containing $Y$. However,
in general, this property may fail: if $X$ is an unshielded continuum
and $Y\subset X$ a closed subset, then it is not necessarily the case
that there exists a unique smallest subcontinuum in $X$ containing $Y$. An
obvious example is a circle $X$ containing a two-point set $Y$.

This motivates the following definition: a subcontinuum $B\subset X$ of an unshielded continuum $X$ is
said to be \emph{complete} if for every bounded component $U$ of $\C\sm
X$ either $B\cap \bd(U)$ is at most a single point or $\bd(U)\subset
B$. For a closed set $Y\subset X$, the smallest \emph{complete} continuum $Z$ with $Y\subset
Z\subset X$ always exists: indeed, two complete continua $Z_1,
Z_2\subset X$ in an unshielded continuum $X$ always have a connected intersection (this follows from
the Mayer--Vietoris exact sequence applied to the unbounded components
of $\C\setminus Z_i$) which implies the claimed.

Let $Q\subset\C$ be an unshielded continuum. An {\em atom} of $Q$ is
either a single point of $Q$ or the boundary of a bounded component
of $\C\sm Q$. A point $x\in Q$ is said to be a \emph{cutpoint} of
$Q$ if removing $x$ from the set $Q$ disconnects it, and an
\emph{endpoint} of $Q$ otherwise. More generally, an atom of $Q$ is
a \emph{cut-atom} of $Q$, if removing it from $Q$ disconnects $Q$.

If there is a continuous self-mapping $f:Q\to Q$, then it is natural to
study the dynamics of cutpoints and cut-atoms of $Q$, in particular the
asymptotic behavior of their orbits, e.g. because cutpoints play a role
in constructing geometric Markov partitions. Even though this study can
be conducted in the general setting outlined above, we now concentrate
on the case of (topological) polynomials.

We want to compare the dynamics of different polynomials with
connected Julia sets. To do so we relate their dynamics to that at
infinity; this is done by means of \emph{laminations} \cite{thu85}.
Laminations are also used to describe a model for the boundary of
$\M_2$, the quotient space $\uc/\qml$ of the unit circle
$\uc\subset\C$ under the so-called \emph{Quadratic Minor Lamination}
$\qml$, a specific equivalence relation on $\uc$. If one takes the
quotient of the unit disk in the plane by collapsing convex hulls of
all classes of $\qml$ to points while not identifying other points,
one gets a model of the entire $\M_2$ (conjecturally, this model is
homeomorphic to $\M_2$). We define laminations below, however our
approach is somewhat different from \cite{thu85} (cf.
\cite{blolev02a}).

\begin{dfn}\label{d:lam}

An equivalence relation $\sim$ on the unit circle $\uc$ is called a
\emph{lamination} if it has the following properties:

\noindent (E1) the graph of $\sim$ is a closed subset in $\uc \times
\uc$;

\noindent (E2) if $t_1\sim t_2\in \uc$ and $t_3\sim t_4\in \uc$, but
$t_2\not \sim t_3$, then the open straight line segments in $\C$ with
endpoints $t_1, t_2$ and $t_3, t_4$ are disjoint;

\noindent (E3) each equivalence class of $\sim$ is totally
disconnected.
\end{dfn}

Consider the map $\si_d:\uc\to\uc$ defined by the formula $\si_d(z)=z^d$.

\begin{dfn}\label{d:si-inv-lam}
A lamination $\sim$ is called ($\si_d$-){\em invariant} if:

\noindent (D1) $\sim$ is {\em forward invariant}: for a class $g$,
the set $\si_d(g)$ is a class too;

\noindent (D2) $\sim$ is {\em backward invariant}: for a class $g$,
its preimage $\si_d^{-1}(g)=\{x\in \uc: \si_d(x)\in g\}$ splits
into at most $d$ classes;

\noindent (D3) for any $\sim$-class $g$, the map $\si_d: g\to
\si_d(g)$ extends to $\uc$ as an orientation preserving covering map
such that $g$ is the full preimage of $\si_d(g)$ under this covering map.

\noindent (D4) all $\sim$-classes are finite.

\end{dfn}

Part (D3) of Definition~\ref{d:lam} has an equivalent version. A
{\em (positively oriented) hole $(a, b)$ of a compactum $Q\subset \uc$}
is a component of $\uc\sm Q$ such that moving from $a$ to $b$
inside $(a, b)$ is in the positive direction. Then (D3) is equivalent
to the fact that for a $\sim$-class $g$ either $\si_d(g)$ is a point or
for each positively oriented hole $(a, b)$ of $g$ the positively
oriented arc $(\si_d(a), \si_d(b))$ is a positively oriented hole of
$\si_d(g)$.

For a $\si_d$-invariant lamination $\sim$ we consider the
\emph{topological Julia set} $\uc/\sim=J_\sim$ and the
\emph{topological polynomial} $f_\sim:J_\sim\to J_\sim$ induced by
$\si_d$. The quotient map $p_\sim:\uc\to J_\sim$ extends to the
plane with the only non-trivial fibers being the convex hulls of
$\sim$-classes. Using Moore's Theorem
one can extend $f_\sim$ to a branched-covering map $f_\sim:\C\to \C$ of
the same degree. The
complement of the unbounded component of $\C\sm J_\sim$ is called
the \emph{filled-in topological Julia set} and is denoted $K_\sim$.
If the lamination $\sim$ is fixed, we may omit $\sim$ from the
notation. For points $a$, $b\in\uc$, let $\ol{ab}$ be the {\em
chord} with endpoints $a$ and $b$ (if $a=b$, set $\ol{ab}=\{a\}$).
For $A\subset\uc$ let $\ch(A)$ be the \emph{convex hull} of $A$ in $\C$.

\begin{dfn}\label{d:lea}
If $A$ is a $\sim$-class, call an edge $\ol{ab}$ of $\bd(\ch(A))$ a
\emph{leaf} (if $a=b$, call the leaf $\ol{aa}=\{a\}$
\emph{degenerate}, cf.\ \cite{thu85}). Also, all points of $\uc$ are called \emph{leaves}. Normally, leaves are denoted
as above, or by a letter with a bar above it ($\bar b, \bq$ etc), or by $\ell$.
The family of all leaves of
$\sim$, denoted by $\lam_\sim$, is called the \emph{geometric
lamination generated by $\sim$}. Denote the union of all leaves of
$\lam_\sim$ by $\lam^+_\sim$. Extend $\si_d$ (keeping the notation)
linearly over all \emph{individual chords} in $\ol{\disk}$, in
particularly over leaves of $\lam_\sim$. Note, that even though the
extended $\si_d$ is not well defined on the entire disk, it is well
defined on $\lam^+$ (and every individual chord in the disk). The
closure of a non-empty component of $\disk\sm \lam^+_\sim$ is called
a \emph{gap} of $\sim$. If $G$ is a gap, we talk about \emph{edges
of $G$}; if $G$ is a gap or leaf, we call the set $G'=\uc\cap G$ the
\emph{basis of $G$}.
\end{dfn}

A gap or leaf $U$ is said to be \emph{preperiodic} if for some
minimal $m>0$ the set $\si_d^m(U')$ is periodic (i.e.,
$\si_d^{m+k}(U')=U'$ for some $k\ge 0$),  and $U', \dots,
\si_d^{m-1}(U')$ are not periodic. Then the number $m$ is called the
\emph{preperiod} of $U$. If $U$ is either periodic or preperiodic,
we will call it \emph{(pre)periodic}.  A leaf $\ell$ is said to be
critical if $\si(\ell)$ is a point. A gap $U$ is said to be
\emph{critical} if $\si|_{\bd(U)}$ is not one-to-one (equivalently,
if $\si_d|_{U'}$ is not one-to-one).
We define \emph{precritical} and \emph{(pre)critical} objects
similarly to the above.

Infinite gaps of a $\si_d$-invariant lamination $\sim$ are called
\emph{Fatou gaps}. Let $G$ be a Fatou gap; by \cite{kiwi02} $G$ is
(pre)periodic under $\si_d$. If a Fatou gap $G$ is periodic, then by
\cite{blolev02a} its basis $G'$ is a Cantor set and the quotient
projection $\psi_G:\bd(G)\to\uc$, which collapses all edges of $G$
to points, is such that $\psi_G$-preimages of points are points or
single leaves.

\begin{dfn}\label{d:degreek}
Suppose that $G$ is a periodic Fatou gap of minimal period $n$.
By \cite{blolev02a} $\psi_G$ semiconjugates $\si^n_d|_{\bd(G)}$ to a map
$\hat\si_G=\hat\si:\uc\to \uc$ so that either {\rm(1)} $\hat\si=\si_k:\uc\to \uc,
k\ge 2$ or {\rm(2)} $\hat\si$ is an irrational rotation. In case (1)
call $G$ a \emph{gap of degree $k$}; if $k=2$, then $G$ is said to be
\emph{quadratic}, and if $k=3$, then $G$ is said to be \emph{cubic}. In
case (2) $G$ is called a \emph{Siegel gap}. A (pre)periodic gap
eventually mapped to a periodic gap of degree $k$ (quadratic, cubic,
Siegel) is also said to be of degree $k$ (quadratic, cubic, Siegel).
\emph{Domains} (bounded components of the complement) of $J_\sim$ are said to
be \emph{of degree $k$, quadratic, cubic and Siegel} if the
corresponding gaps of $\lam_\sim$ are such.
\end{dfn}

Various types of gaps and domains described in
Definition~\ref{d:degreek} correspond to various types of atoms of
$J_\sim$; as with gaps and domains, we keep the same terminology
while replacing the word ``gap'' or ``domain'' by the word ``atom''.
Thus, the boundary of a (periodic) Siegel domain is called a
\emph{(periodic) Siegel atom}, the boundary of a (periodic) Fatou
domain of degree $k>1$ is called a \emph{(periodic) Fatou atom} etc.

An atom $A$ of $J_\sim$ is said to be a \emph{persistent cut-atom}
if all its iterated $f_\sim$-images are cut-atoms. A \emph{persistent cut-atom of
degree $1$} is either a non-(pre)periodic persistent cut-atom, or a
(pre)periodic cut-atom of degree $1$;  (pre)periodic atoms of degree
$1$ are either (pre)periodic points or Siegel gaps (recall, that all
Siegel gaps are (pre)periodic). A \emph{persistent cut-atom of
degree $k>1$} is a Fatou atom of degree $k$.
Let $\ic_{f_\sim}(A)$ (or $\ic(A)$ if $\sim$ is fixed) be the smallest complete \emph{\textbf{i}nvariant \textbf{c}ontinuum in
$J_\sim$} containing a set $A\subset J_\sim$; we call $\ic(A)$ the \emph{dynamical span} of $A$.
A recurring theme in our paper is the fact that in some cases the dynamical span of a certain
set $A$ and the dynamical span of the subset $B\subset A$ consisting of all periodic
elements of $A$ with some extra-properties (e.g., being a periodic cut-atom, a periodic cut-atom of
degree 1 etc) coincide.
Also, call a periodic atom $A$ of period $n$
and degree $1$ \emph{rotational} if $\si^n_d|_{p^{-1}(A)}$ has
\textbf{non-zero rotation number}.
Finally, the \emph{$\om$-limit set} $\om(Z)$ of a set (e.g., a singleton)
$Z\subset J_\sim$ is defined as
$$
\om(Z)=\bigcap_{n=1}^\infty\ol{\bigcup_{i=n}^\infty f_\sim^i(Z)}.
$$

\begin{dfn}\label{dynacore}
The \emph{(dynamical) core} $\ce(f_\sim)$ of $f_\sim$ is the
dynamical span of the union of the $\om$-limit
sets of all persistent cut-atoms. The union of all periodic
cut-atoms of $f_\sim$ is denoted by $\pc(f_\sim)=\pc$ and is called
the \emph{periodic core} of $f_\sim$.

The \emph{(dynamical) core of degree $1$} $\ce_1(f_\sim)$ of
$f_\sim$ is the dynamical span of the
$\om$-limit sets of all persistent cut-atoms of degree $1$. The
union of all periodic cut-atoms of $f_\sim$ of degree $1$ is denoted
by $\pc_1(f_\sim)=\pc_1$ and is called the \emph{periodic core of
degree $1$} of $f_\sim$.

The \emph{(dynamical) rotational core} $\ce_{rot}(f_\sim)$ of
$f_\sim$ is the dynamical span of the $\om$-limit sets of all
wandering persistent cutpoints and all periodic rotational atoms.
The union of all periodic rotational
atoms of $f_\sim$ is denoted by
$\pc_{rot}(f_\sim)=\pc_{rot}$ and is called the \emph{periodic rotational core}
of $f_\sim$.
\end{dfn}

Clearly, $\ce_{rot}\subset \ce_1\subset \ce$ and $\pc_{rot}\subset \pc_1\subset \pc$. The notion of a
core (and even the terminology) is motivated by unimodal dynamics.
Take a unimodal map $f:[0, 1]\to [0, 1]$ such that $f(0)=f(1)=0$ is
a \emph{topologically repelling} fixed point (i.e., all points in a
small neighborhood of $0$ are mapped farther away from $0$) and
denote the unique extremum of $f$ by $c$; to avoid trivialities let
$0<f^2(c)<c<f(c)<1$ and let there be no fixed points in $(0,
f^2(c))$. Then all cutpoints of $[0, 1]$ (i.e., all points of $(0,
1)$) are eventually mapped to the forward invariant interval
$I_f=[f^2(c), f(c)]$, called the \emph{(dynamical) core} of $f$.
Clearly, this is related to Definition~\ref{dynacore}. The so-called
\emph{growing trees} \cite{blolev02a} are also related to the notion
of the dynamical core.

One of the aims of our paper is to illustrate the analogy between
the dynamics of topological polynomials on their \emph{cutpoints}
and \emph{cut-atoms} and interval dynamics. E.g., it is known, that
for interval maps periodic points and critical points  play a
significant, if not decisive, role. Similar results can be obtained
for topological polynomials too. Observe, that in the case of
dendrites the notions become simpler and some results can be
strengthened. Indeed, first of all in that case we can talk about
cutpoints only. Second, in this case $\ce=\ce_1$ and $\pc=\pc_1$. A
priori, in that case $\ce_{rot}$ could be strictly smaller than
$\ce$, however Theorem~\ref{t:criticoreintr} shows that in this case
$\ce_{rot}=\ce$.


\begin{thm}\label{t:criticoreintr}
The dynamical core of $f_\sim$ coincides with $\ic(\pc(f_\sim))$.
The dynamical core of degree $1$ of $f_\sim$ coincides with
$\ic(\pc_1(f_\sim))$. The rotational dynamical core coincides with
$\ic(\pc_{rot}(f_\sim))$. If $J_\sim$ is a dendrite, then
$\ce=\ic(\pc_{rot}(f_\sim))$.
\end{thm}

There are two other main results in Section~\ref{s:dyco}. As the
dynamics on Fatou gaps is simple, it is natural to consider the
dynamics of gaps/leaves which never map to Fatou gaps, or the
dynamics of gaps/leaves which never map to
`maximal concatenations'' of Fatou gaps which we call
\emph{super-gaps} (these notions are made precise in
Section~\ref{s:dyco}). We prove that the dynamical span of limit
sets of all persistent cut-atoms which never map to the $p_\sim$-images of
super-gaps coincides with the dynamical span of all periodic
rotational cut-atoms located outside the $p_\sim$-images of super-gaps. In
fact, the ``dendritic'' part of Theorem~\ref{t:criticoreintr}
follows from that result.

A result similar to Theorem~\ref{t:criticoreintr}, using critical
points and atoms instead of periodic ones, is proven in Theorem~\ref{t:critcore}. Namely, an atom $A$ is \emph{critical} if either $A$ is a critical point of
$f_\sim$, or $f_\sim|_A$ is not one-to-one.
In Theorem~\ref{t:critcore} we prove, in particular, that
various cores of a topological polynomial $f_\sim$ coincide with the dynamical span
of critical atoms of the restriction of $f_\sim$ onto these cores.

If $J_\sim$ is a dendrite, we can replace critical atoms in $\ce(f_\sim)$
by critical points; this is closely related to the interval, even unimodal, case.
Thus, Theorem~\ref{t:criticoreintr} can be viewed as a generalization of the
corresponding results for maps of the interval: if $g$ is a
piecewise-monotone interval map then the closure of the union of the
limit sets of all its points coincides with the closure of the union of
its periodic points (see, e.g., \cite{blo95}, where this is deduced
from similar results which hold for all continuous interval maps, and
references therein).

Theorem~\ref{t:criticoreintr} shows the importance of the periodic cores of $f_\sim$.
In the second half of the paper we concentrate on the cubic case and
describe cubic topological polynomials such that their laminations
have only isolated leaves (such laminations have a dynamical
rotational core consisting of at most one atom which, by
Theorem~\ref{t:criticoreintr}, is equivalent to having periodic
rotational core of no more than one atom).
In fact, it is easy to show that the quadratic laminations with at most one periodic rotational
atom are exactly those with only isolated leaves, and furthermore, such laminations form
the Main Cardioid $\car_2$.
The corresponding description in the cubic case is more involved,
see Theorem \ref{t:cormin-spec}.

\subsubsection*{Organization of the paper}
In Section 2, we review the terminology of geometric laminations, due
to Thurston, and discuss some general properties of invariant
laminations.
Section 3 contains the explicit descriptions of the dynamical core,
in particular, the proof of Theorem \ref{t:criticoreintr}.
In Section 4, we classify invariant quadratic gaps of cubic laminations.
With every such gap, we associate its so-called \emph{canonical lamination}.
Section 5 proceeds with a similar study of invariant rotational gaps.
Finally, in Section 6, we describe all cubic invariant laminations with
only isolated leaves.

\subsubsection*{Acknowledgements}
During the work on this project, the fourth author has visited Max
Planck Institute for Mathematics (MPIM), Bonn. He is very grateful to
MPIM for inspiring working conditions.

\section{Preliminaries}\label{s:prel}

In this section, we go over well-known facts concerning laminations. We
also quote some topological results used in what follows.

Let $\bbd$ be the open unit disk and $\hc$ be the complex sphere.
For a compactum $X\subset\C$, let $\iU(X)$ be the unbounded
component of $\C\sm X$. The topological hull of $X$ equals
$\tha(X)=\C\sm \iU(X)$ (often we use $\iU(X)$ for $\hc\sm\tha(X)$,
including the point at $\infty$). If $X$ is a continuum, then
$\tha(X)$ is a \emph{non-separating} continuum, and there exists a
Riemann map $\Psi_X:\hc\sm\ol\bbd\to \iU(X)$; we always normalize it
so that $\Psi_X(\infty)=\infty$ and $\Psi'_X(z)$ tends to a positive
real limit as $z\to\infty$.

Consider a polynomial $P$ of degree $d\ge 2$ with Julia set $J_P$
and filled-in Julia set $K_P=\tha(J_P)$. Clearly, $J_P$ is
unshielded. Extend $z^d: \C\to \C$ to a map $\ta_d$ on $\hc$. If $J_P$
is connected then $\Psi_{K_P}=\Psi:\C\sm\ol\bbd\to \iU(K_P)$ is such
that $\Psi\circ \ta_d=P\circ \Psi$ on the complement to the closed unit
disk \cite{hubbdoua85,miln00}. If $J_P$ is locally connected, then
$\Psi$ extends to a continuous function $\ol{\Psi}: {\hc\sm\bbd}\to
\ol{\hc\setminus K_P}$, and $\ol{\Psi} \circ\, \ta_d=P\circ\ol{\Psi}$
on the complement of the open unit disk; thus, we obtain a continuous
surjection $\ol\Psi\colon\bd(\bbd)\to J_P$ (the \emph{Carath\'eodory loop}).
Identify $\uc=\bd(\bbd)$ with $\mathbb{R}/\mathbb{Z}$.

Let $J_P$ be locally connected, and set $\psi=\ol{\Psi}|_{\uc}$.
Define an equivalence relation $\sim_P$ on
$\uc$ by $x \sim_P y$ if and only if $\psi(x)=\psi(y)$, and call it the
($\si_d$-invariant) {\em lamination of $P$}.
Equivalence classes of $\sim_P$ are pairwise \emph{unlinked}:
their Euclidian convex hulls are disjoint.
The topological Julia set $\uc/\sim_P=J_{\sim_P}$ is homeomorphic to $J_P$,
and the topological polynomial $f_{\sim_P}:J_{\sim_P}\to J_{\sim_P}$ is topologically conjugate to $P|_{J_P}$.
One can extend the conjugacy between
$P|_{J_{P}}$ and $f_{\sim_P}:J_{\sim_P}\to J_{\sim_P}$ to a
conjugacy on the entire plane.

An important
particular case is that when $J_\sim$ is a \emph{dendrite} (a locally
connected continuum containing no simple closed curve) and so
$\hc\sm J_\sim$ is a simply connected neighborhood of infinity. It is easy to
see that if a lamination $\sim$ has no domains (in other words, if
convex hulls of all $\sim$-classes partition the entire unit disk),
then the quotient space $\uc/\sim$ is a dendrite.

\subsection{Geometric laminations}\label{ss:geol}

The connection between laminations, understood as equivalence relations, and the
original approach of Thur\-ston's \cite{thu85}, can be explained once we
introduce a few key notions. Assume that a $\si_d$-invariant lamination
$\sim$ is given.

Thurston's idea was to study similar collections of chords in $\disk$
abstractly, i.e., without assuming that they are generated by an
equivalence relation on the circle with specific properties.

\begin{dfn}[\rm{cf.} \cite{thu85}]\label{geolam}
A \emph{geometric prelamination} $\lam$ is a set of (possibly degenerate) chords
in $\ol{\disk}$ such that any two distinct chords from $\lam$ meet at
most in an endpoint of both of them; $\lam$ is called a \emph{geometric
lamination} (\emph{geo-lamination}) if all points of $\uc$ belong to
elements of $\lam$, and $\bigcup \lam$ is closed.
Elements of $\lam$ are called \emph{leaves} of $\lam$
(thus, some leaves may be degenerate).
The union of all leaves of $\lam$ is denoted by $\lam^+$.
\end{dfn}

\begin{dfn}\label{d:lamgap}
Suppose that $\lam$ is a geo-lamination. The closure of a non-empty
component of $\disk\sm \lam^+$ is called a \emph{gap} of $\lam$.
Thus, given a geo-lamination $\lam$ we obtain a cover of
$\ol{\disk}$ by gaps of $\lam$ and (perhaps, degenerate) leaves of
$\lam$ which do not lie on the boundary of a gap of $\lam$
(equivalently, are not isolated in $\disk$ from either side).
Elements of this cover are called {\em $\lam$-sets}.
Observe that the intersection of two different $\lam$-sets is at most a leaf.
\end{dfn}

In the case when $\lam=\lam_\sim$ is generated by an invariant
lamination $\sim$, gaps might be of two kinds: finite gaps which are convex
hulls of $\sim$-classes and infinite gaps which are closures of
\emph{domains} (of $\sim$), i.e. components of $\disk\sm \lam^+_\sim$
which are not interiors of convex hulls of $\sim$-classes.

\begin{dfn}[\rm{cf.} \cite{thu85}]\label{geolaminv}
A geometric (pre)lamination $\lam$ is said to be an \emph{invariant}
geo-lamination of degree $d$ if the following conditions are
satisfied:

\begin{enumerate}

\item (Leaf invariance) For each leaf $\ell\in \lam$, the set
    $\si_d(\ell)$ is a (perhaps degenerate) leaf in $\lam$. For every
     non-degenerate leaf $\ell\in\lam$, there are $d$ pairwise disjoint
    leaves $\ell_1,\dots,\ell_d$ in $\lam$ such that for each $i$,
    $\si_d(\ell_i)=\ell$.

\item (Gap invariance) For a gap $G$ of $\lam$, the set
    $H=\ch(\si_d(G'))$ is a (possibly degenerate) leaf, or a gap of
    $\lam$, in which case $\si_d|_{\bd(G)}:\bd(G)\to \bd(H)$
    (by Definition~\ref{d:lea} $\si_d$ is defined on $\bd(G)$) is a
    positively oriented composition of a monotone map and a
    covering map (a {\em monotone map} is a map such that the full
    preimage of any connected set is connected).
\end{enumerate}

\end{dfn}

Note that some geo-laminations are not generated by equivalence
relations. We will use a special extension $\si^*_{d, \lam}=\si_d^*$
of $\si_d$ to the closed unit disk associated with $\lam$. On
$\uc$ and all leaves of $\lam$, we set $\si^*_d=\si_d$ (in
Definition~\ref{d:lea}, $\si_d$ was extended over all
chords in $\ol \D$, including leaves of $\lam$). 
Otherwise, define $\si^*_d$ on the interiors of gaps using a
standard barycentric construction (see \cite{thu85}). Sometimes we
lighten the notation and use $\si_d$ instead of $\si^*_d$.
We will mostly use the map $\si_d^*$ in the case $\lam=\lam_\sim$ for some invariant
lamination $\sim$.

\begin{dfn}\label{d:crit}
A leaf of a lamination $\sim$ is called \emph{critical} if its
endpoints have the same image. A $\lam_\sim$-set $G$ is said to be
\emph{critical} if $\si_d|_{G\,'}$ is $k$-to-$1$ for some $k>1$. 
E.g., a periodic Siegel gap is a non-critical $\sim$-set, on
whose basis the first return map is not one-to-one because there must
be critical leaves in the boundaries of gaps from its orbit.
\end{dfn}

We need more notation. Let $a, b\in \uc$. By $[a, b], (a, b)$ etc we
denote the appropriate \emph{positively oriented} circle arcs from
$a$ to $b$. 
By $|I|$ we denote the length of
an arc $I$.

\begin{lem}\label{l:ficri}
If $\bq$ is a leaf of a $\si_d$-invariant geo-lamination $\lam$ and
$I$ is a closed circle arc with the same endpoints as $\bq$ such that
$|I|\ge \frac1{d}$, then either $\bq$ is critical, or there exists a
set $G$ with $G'\subset I$ such that $G$ is either a critical
$\lam$-set or a critical leaf.
\end{lem}

Note that a critical leaf may not be an $\lam$-set.

\begin{proof}
Let $\bq=\ol{ab}$ and $I=[a, b]$. Assume that there are no critical
leaves among those with endpoints in $I$; then there exists $k>0$ with $\frac{k-1}d<|I|<\frac{k}d$.
For each leaf $\ol{xy}$ whose hole $(x,y)$ is contained in $I$ let $\la(\ol{xy})$ be the minimal distance from $|[x, y]|$
to the set $T=\{\frac1d, \frac2d, \dots, \frac{k-1}d\}$. Clearly, there
is a leaf $\ell=\ol{xy}$ such that $\la(\ell)=\ga>0$ is minimal, and
its well-defined side $S$ such that any chord $\hell$, close to $\ell$ from
that side, has $\la(\hell)<\la(\ell)$. Hence $\ell$ is an edge of a gap
$G$ located on this very side of $\ell$. Moreover, $G$ is located in
the component of $\disk\sm \ol{ab}$ containing $\ell$. Indeed, this is
clear if $\ell\ne \bq$. On the other hand, if $\ell=\bq$ then any leaf
$\ell_1$ with endpoints outside $[a, b]$ gives $\la(\ell_1)>\la(\ell)$
because we measure the distance from $|[x, y]|$ to the set $\{\frac1d,
\frac2d, \dots, \frac{k-1}d\}$ \emph{but not to} $\frac{k}d$.

Assume first that $|[x, y]|=\frac{i}{d} + \ga$, where $1\le i\le
k-1$. Then the side $S$ of $\ell$ is the one not facing $\bq$, but
facing the circle arc $[x, y]$. Hence, $\si_d(G')\subset [\si(x),
\si(y)]$. Assume that $G$ is not critical. Then it maps onto its
image one-to-one, preserving orientation. Thus the points of $G'$
belong to $\si_d^{-1}[\si_d(x),\si_d(y)]$, but not to the same
component of this set. This follows since
$\si^{-1}_d([\si_d(x),\si_d(y)])$ consists of $d$ disjoint intervals
of length $\ga\le \frac{1}{2d}$ and $|[x, y]|>\frac1d$.
Choose an edge $\ell'=\ol{x'y'}$ of $G$ connecting points from
distinct components of $\si_d^{-1}[\si_d(x),\si_d(y)]$ so that $[x',
y']$ is a positively oriented arc. It follows that $|[x',
y']|>\frac{1}d$ while $|[\si_d(x'), \si_d(y')]|<|[\si_d(x),
\si_d(y)]|$, contradicting the choice of $\ol{xy}$. The case when
$[x, y]=\frac{i}{d} - \ga, 1\le i\le k-1$, can be considered
similarly.
\end{proof}

\subsection{Laminational sets and their basic properties}\label{ss:st}

\begin{dfn}\label{d:restuff}
Let $f:X\to X$ be a self-mapping of a set $X$. For a set $G\subset
X$, define the \emph{return time} (to $G$) of $x\in G$ as the least
positive integer $n_x$ such that $f^{n_x}(x)\in G$, or infinity if there is
no such integer. Let $n=\min_{y\in G} n_y$, define $D_G=\{x:n_x=n\}$,
and call the map $f^n:D_G\to G$
the \emph{remap} (of $G$).
Similarly, we talk about \emph{refixed}, \emph{reperiodic} points (of a certain \emph{reperiod}),
and \emph{reorbits} of points in $G$.
\end{dfn}

E.g., if $G$ is the boundary of a periodic Fatou domain of period
$n$ of a topological polynomial $f_\sim$ whose images are all
pairwise disjoint until $f^n_\sim(G)=G$, then $D_G=G$ and the
corresponding remap on $D_G=G$ is the same as $f^n_\sim$.

\begin{dfn}\label{d:lamset}
If $A\subset \uc$ is a closed set such that all the sets
$\ch(\si^i(A))$ are pairwise disjoint, then $A$ is called
\emph{wandering}. If there exists $n\ge 1$ such that all the sets
$\ch(\si^i(A)), i=0, \dots, n-1$ have pairwise disjoint interiors while
$\si^n(A)=A$, then $A$ is called \emph{periodic}
of period $n$. If there exists
$m>0$ such that all $\ch(\si^i(A)), 0\le i\le m+n-1$ have pairwise
disjoint interiors and $\si^m(A)$ is periodic
of period $n$, then we call $A$
\emph{preperiodic}. Moreover, suppose that $A$ is wandering, periodic
or preperiodic, and for every $i\ge 0$ and every hole $(a, b)$ of
$\si^i(A)$ either $\si(a)=\si(b)$ or the positively oriented arc
$(\si(a), \si(b))$ is a hole of $\si^{i+1}(A)$.
Then we call $A$ (and $\ch(A)$) a \emph{laminational set};
we call both $A$ and $\ch(A)$ \emph{finite} if $A$ is finite.
A {\em stand alone gap} is defined as a laminational set with
non-empty interior.
\end{dfn}

We say that a closed set $A\subset\uc$ (and its convex hull) is
a {\em semi-laminational set} if, for every hole $(x,y)$
of $A$, we have $\si_d(x)=\si_d(y)$, or the open arc $(\si_d(x),\si_d(y))$
is a hole of $A$.
Note that we do not assume that $\si_d(A)=A$ (if this holds, then
$A$ is an invariant laminational set).

Recall that in Definition~\ref{d:degreek} we defined Fatou gaps $G$
of various degrees as well as Siegel gaps. Given a periodic Fatou
gap $G$ we also introduced the monotone map $\psi_G$ which
semiconjugates $\si_d|_{\bd(G)}$ and the appropriate model map
$\hat\si_G:\uc\to \uc$.
This construction can be also done for stand alone Fatou gaps $G$.

Indeed, consider the basis $G\cap \uc=G'$ of $G$ (see
Definition~\ref{d:lea}) as a subset of $\bd(G)$.
It is well-known that $G'$ coincides with the union $A\cup B$ of two well-defined
sets, where $A$ is a Cantor subset of $G'$ or an empty set and $B$ is a countable set.
In the case when $A=\0$, the map $\psi_G$ simply collapses $\bd(G)$ to a point.
However, if $A\ne \0$, one can define a semiconjugacy $\psi_G:\bd(G)\to \uc$ which collapses all holes of
$G'$ to points. As in Definition~\ref{d:degreek}, the map $\psi_G$
semiconjugates $\si_d|_{\bd(G)}$ to a circle map which is either an
irrational rotation or the map $\si_k, k\ge 2$. Depending on the
type of this map we can introduce for periodic infinite laminational
sets terminology similar to Definition~\ref{d:degreek}. In
particular, if $\si_d|_{\bd(G)}$ is semiconjugate to $\si_k, k\ge 2$
we say that $G$ is a \emph{stand alone Fatou gap of degree $k$.}

\begin{dfn}\label{d:relam}
If $G$ is a periodic laminational set such that $G'$ is finite and
contains no refixed points, then $G$ is said to be a \emph{finite}
\emph{(periodic) rotational set}. Finite rotational sets and Siegel
gaps $G$ are called \emph{(periodic) rotational sets}. If such $G$
is invariant, we call it an \emph{invariant rotational set}.
\end{dfn}

The maps $\si_k$ serve as models of remaps on periodic gaps of
degree $k\ge 2$. For rotational sets, models of remaps are
rotations.

\begin{dfn}\label{d:rotnum}
A number $\tau$ is said to be the \emph{rotation number} of a
periodic set $G$ if for every $x\in G'$ the circular order of
points in the reorbit of $x$ is the same as the order of points $0,
\rot(0), \dots$ where $\rot:\uc\to \uc$ is the rigid rotation by the
angle $\tau$.
\end{dfn}

It is easy to see that to each rotational set $G$, one can associate
its well-defined rotation number $\tau_G=\tau$ (in the case of a
finite rotational set, the property that endpoints of holes are
mapped to endpoints of holes implies that the circular order on $G'$
remains unchanged under $\si$). Since $G'$ contains no refixed
points by our assumption, $\tau\ne 0$. Given a topological polynomial $f_\sim$ and
an $f_\sim$-periodic point $x$ of period $n$,
we can associate to $x$ the rotation number $\rho(x)$ of $\si^n_d$ restricted to the
$\sim$-class $p_\sim^{-1}(x)$, corresponding to $x$ (recall, that
$p_\sim:\uc\to J_\sim=\uc/\sim$ is the quotient map associated to
$\sim$); then if $\rho(x)\ne 0$ the set $p_\sim^{-1}(x)$ is
rotational, but if $\rho(x)=0$ then the set $p_\sim^{-1}(x)$ is not
rotational.

The following result allows one to find fixed laminational sets of
specific types in some parts of the disk; for the proof see \cite{bfmot10}.

\begin{thm}\label{t:fxpt}
Let $\sim$ be a $\si_d$-invariant lamination. Consider the
topological polynomial $f_\sim$ extended over $\C$. Suppose that
$e_1, \dots, e_m\in J_\sim$ are $m$ points, and $X\subset K_\sim$ is
a component of $K_\sim\sm \{e_1, \dots, e_m\}$ such that for each
$i$ we have $e_i\in \ol{X}$ and either $f_\sim(e_i)=e_i$ and the
rotation number $\rho(e_i)=0$, or $f_\sim(e_i)$ belongs to the
component of $K_\sim\sm \{e_i\}$ which contains $X$. Then at least
one of the following claims holds:

\begin{enumerate}

\item $X$ contains an invariant domain of degree $k>1$;

\item $X$ contains an invariant Siegel domain;

\item $X\cap J_\sim$ contains a fixed point with non-zero rotation
number.

\end{enumerate}

\end{thm}

Observe that for dendritic topological Julia sets $J_\sim$ the claim
is easier as cases (1) and (2) above are impossible. Thus, in the
dendritic case Theorem~\ref{t:fxpt} implies that there exists a
rotational fixed point in $X\cap J_\sim$.

Let $U$ be a laminational set. For every edge $\ell$ of $U$, let
$H_U(\ell)$ denote the hole of $U$ that shares both endpoints with
$\ell$ (if $U$ is fixed, we may drop the subscript in the above
notation). This is well-defined if $U$ is a gap. If $U=\ell$ is a
leaf, then there are two arc-components of $\uc\sm U$, and each
time $H_U(\ell)$ is needed, we explain which arc we mean.
The hole $H_U(\ell)$ is called the \emph{hole of $U$ behind (at) $\ell$}.
In this situation we define $|\ell|_U$ as $|H_U(\ell)|$.
It is also convenient to consider similar notions for a closed connected union $A$ of laminational sets.
Given a leaf $\ell\in A$, it may happen that $A$
is contained in the closure of a unique component of $\disk\sm \ell$.
Then the other component of $\disk\sm \ell$ is bounded by $\ell$ and an arc open in $\uc$,
called the \emph{hole of $A$ behind (at) $\ell$}, and denoted by $H_A(\ell)$.
As before, set $|\ell|_A=|H_A(\ell)|$.
Observe that distinct holes of $A$ are pairwise disjoint.

Suppose that $G$ is a laminational set. Mostly, the holes of $G$ map
increasingly onto the holes of $\si_d(G)$. However, if the length of
a hole is at least $\frac1{d}$, a different phenomenon takes place.
Namely, suppose that there is an edge $\ol{ab}$ of $G$ such that the
open arc $(a,b)$ is a hole of $G$. Suppose that $\frac kd<|(a,
b)|<\frac {k+1}d$. By definition, $(\si_d(a),\si_d(b))$ is a hole of
$\si_d(G)$. The map $\si_d$ maps $(a,b-\frac kd)$ onto
$(\si_d(a),\si_d(b))$ increasingly and then wraps $[b-\frac kd, b)$
around the circle exactly $k$ times.
This shows the importance of holes
that are longer than $\frac 1d$ as these are the only holes which
the map $\si_d$ does not take one-to-one onto the ``next'' hole.

\begin{dfn}\label{d:0major}
If $\ell$ is an edge of a $\si_d$-laminational set $G$ such that its
hole $H_G(\ell)$ is not shorter than $\frac 1d$, then $\ell$ is
called a \emph{major edge of $G$ (with respect to $\si_d$)} (or
simply a \emph{major of $G$ (with respect to $\si_d$)}, and
$H_G(\ell)$ is called a \emph{major hole of $G$ (with respect to
$\si_d$)}.
\end{dfn}

Observe that some leaves which are not majors of $G$ with respect to
$\si_d$ may become majors of $G$ with respect to higher powers of $\si_d$.

\begin{lem}\label{l:maj}
{\rm(1)} If $\ell$ is any chord, then there exists a component of $\uc\sm
\ell$ of length at least $\frac 12\ge \frac{1}{d}$.

{\rm(2)} Suppose that $\ell=\ol{xy}$ is a non-invariant leaf of an
invariant semi-laminational set $Q$.
Then there exists $n$ such that $H_Q(\si_d^n(\ell))$ is of length
at least $\frac{1}d$, and $\ell$ is either (pre)critical, or (pre)periodic.
In particular, an edge of a (pre)periodic infinite
gap of a geo-lamination is either (pre)periodic or (pre)critical.

{\rm(3)} For any edge $\ell$ of a laminational set $G$ there exists
$n$ such that $\si_d^n(\ell)$ is a major of $\si_d^n(G)$. If $G$ is
periodic, then $\ell$ is either (pre)critical, or (pre)periodic.
\end{lem}

\begin{proof}
(1) Left to the reader.

(2) For any $i$, set $T_i=H_Q(\si^i_d(\ell))$.
If $|T_i|<\frac{1}d$ then $|T_{i+1}|=|\si_d(T_i)|=d|T_i|$ proving the first claim of (2).

Suppose that $\ell$ is not precritical. Then all arcs $T_i$ are
non-degenerate. Hence in this case we can divide the orbit of $\ell$
into segments, each ending with a leaf $\si_d^n(\ell)$ such that
$|T_n|>\frac{1}d$. Since there cannot be more than $d$ such arcs $T_n$,
we conclude that $\ell$ must be (pre)periodic.

(3) If $\ell$ is precritical, choose $n$ such that $\si_d^n(\ell)$
is a critical leaf; clearly, then $\si_d^n(\ell)$ is a major of
$\si^n_d(G)$. Otherwise suppose first that $G$ is such that for some
minimal $t$ the set $(\si_d^*)^t(G)$ is a leaf. Then it suffices to
observe that there exists a component of $\uc\sm (\si_d^*)^t(G)$
which is longer than or equal to $\frac{1}d$ and that we can take
$n=t$. So, we may assume that $\ell$ is non-(pre)critical and no image
of $G$ is a leaf. As in (2), we observe that if $\ell$ is not a
major, then

$$|H_{(\si_d^*)^{n+1}(G)}(\si_d^{n+1}(\ell))|=|\si^*_d(H_{(\sigma_d^*)^n(G)}(\sigma_d^n(\ell))|
=d|H_{(\sigma_d^*)^n(G)}(\sigma_d^n(\ell))|.$$

\noindent This proves the first claim. The second claim follows from
the same arguments as used in the proof of (2).
\end{proof}

\begin{lem}
\label{l:fx-maj} An edge of an invariant laminational set $G$ is a
major if and only if the closure of its hole contains a fixed point.
Moreover, an edge of an invariant laminational set $G$ which does
not have fixed vertices is a major if and only if its hole contains
a fixed point.
\end{lem}

\begin{proof}
Let $\ell$ be an edge of $G$. The case when $\ell$ has fixed
endpoints is left to the reader. Otherwise if $|H_G(\ell)|<1/d$,
then $H_G(\ell)$ maps onto the hole $\si_d(H_G(\ell))$ one-to-one.
Since in this case $\si_d(H_G(\ell))$ is disjoint from $H_G(\ell)$,
we see that $H_G(\ell)$ contains no fixed points. On the other hand,
suppose that $|H_G(\ell)|\ge 1/d$. Then $\si_d(H_G(\ell))$ covers
the entire $\uc$ while the images the endpoints of $H_G(\ell)$ are
outside $H_G(\ell)$. This implies that there exist a fixed point $a$
in $\ol{H_G(\ell)}$, and by the assumption we see that $a\in
\ol{H_G(\ell)}$.
\end{proof}


If $G$ is a stand alone Fatou gap, then the map $\psi_G$ is defined
on the boundary of $G$. If $\ell$ is any chord connecting two points
$x,y\in G'$, then we define the {\em $\psi_G$-image} of $\ell$ as
the chord connecting $\psi_G(x)$ with $\psi_G(y)$ (this chord may be
degenerate).

\begin{lem}
\label{l:inv-subgap} Suppose that $G$ is a stand alone Fatou gap of
degree $k>1$ and period $n$. Suppose also that $\sim$ is a
$\si_d$-invariant lamination. If no leaf of $\sim$ crosses any major
of $G$ (understood as majors of $G$ with respect to $\si_d^n$)
inside $\disk$, then the $\psi_G$-images of leaves of $\sim$ form an
invariant geo-lamination $\lam_{\psi_G(\sim)}$ of degree $k$.
\end{lem}

\begin{proof}
 Observe that $\psi_G$ semi-conjugates the restriction of $\si^n_d$
to $\bd(G)$ with the appropriate $\si_k$. We want to apply the
projection $\psi_G$ to the leaves of $\sim$ which cross $G$.

It follows from the definition of a major that if a leaf $\bq$ of
$\sim$ intersects (in $\disk$) two non-major edges $\ell$, $\ell''$ of $G$ that
have distinct images under $\si_d^n$ then $\si^n_d(\bq)$ intersects
$\si^n_d$-images of $\ell$, $\ell''$. By Lemma~\ref{l:maj}(3) both
$\ell$ and $\ell''$ will eventually map to majors of $G$; on the
other hand, no leaf of $\sim$ (in particular, no image of $\bq$) can
intersect majors of $G$ in $\disk$. Hence on some step the appropriate image of
$\bq$ will have to intersect two edges of $G$ with equal images.
This implies that the $\psi_G$-projection of this image of $\bq$ is a
$\si_k$-critical leaf. Otherwise leaves of $\sim$ cannot cross edges
of $G$ in $\disk$.

It is easy to see that this implies that the $\psi_G$-images of
leaves of $\sim$ intersecting $G$ form a geo-lamination. The
properties of geo-laminations generated by laminations, and the
description of the way leaves of $\sim$ can intersect edges of $G$
given in the previous paragraph now show that the induced in $\uc$
geo-lamination $\psi_G(\lam_\sim)$ is $\si_k$-invariant.
\end{proof}

\subsection{Cubic laminations}

By cubic laminations we mean $\si_3$-invariant laminations.
By definition, the rotation number of a periodic rotational set is not $0$.
Thus, the horizontal diameter-leaf $\di=\ol{0\frac12}$ is not a rotational set
(we think of $0$ and $\frac 12$ as being elements of $\R/\mathbb{Z}=\uc$; if
we realize $\uc$ as the unit circle in the plane of complex numbers, then
$0$ and $\frac 12$ will turn into $1$ and $-1$).
However, we often think of $\di$ along the same lines as when we consider
periodic (invariant) rotational sets.

\begin{dfn}\label{d:types23}
The major of an invariant rotational set $G$ whose major hole
contains $0$ is called the \emph{\mz} (of $G$).
The major of $G$ whose major hole contains $\frac12$ is said to be
the \emph{\mh} (of $G$).
\end{dfn}

\section{Dynamical core}\label{s:dyco}

In Section~\ref{s:dyco} we fix $\sim$ which is a $\si_d$-invariant
lamination, study the dynamical properties of the topological
polynomial $f_\sim:J_\sim\to J_\sim$, and discuss its dynamical core
$\ce_{f_\sim}$.
For brevity, we write $f$, $J$, $p$, $\ce$ etc for $f_\sim$, $J_\sim$, $p_\sim:\uc\to J$, $\ce_{f_\sim}$, respectively,
throughout Section~\ref{s:dyco}.
Note that, by definition, every topological Julia set $J$ in this section is locally connected.

\subsection{Proof of Theorem \ref{t:criticoreintr}}
Lemma~\ref{l:crainva} studies intersections between
atoms and complete invariant continua.


\begin{lem}\label{l:crainva}
Let $X\subset J$ be an invariant complete continuum and $A$ be an atom
intersecting $X$ but not contained in $X$.
Then $A\cap X=\{x\}$ is a singleton, $A$ is the boundary of a Fatou domain,
and one of the following holds:
{\rm(1)} for some $k$ we have $f^i(A)\cap X=\{f^i(x)\}, i<k$ and
$f^k(A)\subset X$, or {\rm(2)} $f^i(A)\cap X=\{f^i(x)\}$ for all $i$
and there exists the smallest $n$ such that $f^n(A)$ is the boundary of a periodic
Fatou domain of degree $r>1$ with $f^n(x)$ being a refixed point of
$f^n(A)$.
\end{lem}

\begin{proof}
Since $X$ is complete, we may assume that $A\cap X=\{x\}$ is a
singleton and $A$ is the boundary of a Fatou domain such that
$f^i(A)\cap X=\{f^i(x)\}$ for all $i$. Choose the smallest $n$ such
that $f^n(A)$ is periodic. If $f^n(x)$ is not refixed in $f^n(A)$,
then another point from the orbit of $x$ belongs to $f^n(A)\cap X$,
a contradiction.
\end{proof}



In the next lemma we study one-to-one maps on complete continua. To
do so we need a few definitions. The set of all critical points of
$f$ is denoted by $\cri_f=\cri$. The $p$-preimage of $\cri$ is
denoted by $\cri_\sim$. We also denote the $\om$-limit set of $\cri$
by $\om(\cri)$ and its $p$-preimage by $\om(\cri_\sim)$. A {\em
critical atom} is the $p$-image of a critical gap or a critical leaf
of $\lam_\sim$. Thus, the family of critical atoms includes all
critical points of $f$ and boundaries of all bounded components of
$\C\sm J$ on which $f$ is of degree greater than $1$ while
boundaries of Siegel domains are not critical atoms. An atom of $J$ is said to be
\emph{precritical} if it eventually maps to a critical atom.

\begin{lem} \label{l:easytop}
Suppose that $X\subset J$ is a complete continuum containing no
critical atoms. Then $f|_X$ is one-to-one.
\end{lem}

\begin{proof}
Otherwise choose points $x, y\in X$ with $f(x)=f(y)$ and connect $x$
and $y$ with an arc $I\subset X$. If there are no Fatou domains with
boundaries in $X$, then $I$ is unique.
Otherwise for each Fatou domain $U$ with $\bd(U)\subset X$, separating $x$ from $y$ in $X$,
there are two points $i_U, t_U\in \bd(U)$, each of which separates
$x$ from $y$ in $X$, such that $I$ must contain one of the two
subarcs of $\bd(U)$ with endpoints $i_U, t_U$. Note that $f(I)$ is
not a dendrite since otherwise there must exist a critical point of
$f|_I$.
Hence we can choose a minimal subarc $I'\subset I$ so that $f(I')$
is a closed Jordan curve (then $f|_{I'}$ is one-to-one except for
the endpoints $x'$, $y'$ of $I'$ mapped into the same point).

It follows that $f(I')$ is the boundary of a Fatou domain $U$
(otherwise there are points of $J$ ``shielded'' from infinity by points
of $f(I')$ which is impossible).
The set $f^{-1}(U)$ is a finite union of Fatou domains, and
$I'$ is contained in the boundary of $f^{-1}(U)$.
Therefore, there are two possible cases. Suppose that $I'$ is a finite concatenation of at least two arcs, each
arc lying on the boundary of some component of $f^{-1}(U)$. Then,
as $I'$ passes from one boundary to another, it must pass
through a critical point, a contradiction. Second, $I'$ is contained
in the boundary of a single component $V$ of $f^{-1}(U)$ which implies that $\bd(V)\subset X$ is a critical atom,
a contradiction.
\end{proof}

Clearly, (pre)critical atoms are dense. In fact, they are also dense
in a stronger sense. To explain this, we need the following
definition. For a topological space $X$, a set $A\subset X$ is
called \emph{continuum-dense} (or {\em condense}) in $X$ if $A\cap Z
\ne \0$ for each non-degenerate continuum $Z\subset X$ (being
non-degenerate means containing more than one point). The notion was
introduced in \cite{bot06} in a different context.

Also, let us introduce a relative version of the notion of a critical atom.
Namely, let $X\subset J$ be a complete continuum.
A \emph{point} $a\in X$ is \emph{critical with respect to $X$} if in an neighborhood
$U$ of $a$ in $X$ the map $f|_U$ is not one-to-one.
An \emph{atom} $A\subset X$ is \emph{critical with respect to $X$} if it is either a
critical point with respect to $X$ or a Fatou atom $A\subset X$ of degree greater
than 1.

\begin{lem}\label{l:condens}
Let $I\subset X\subset J$ be two non-degenerate continua
such that $X$ is complete and invariant. Suppose that $X$ is not a Siegel atom.
Then $f^n(I)$ contains a critical atom of $f|_X$ for some $n$.
Thus, (pre)critical atoms are condense in $J$ (every continuum in $J$ intersects
a (pre)critical atom).
\end{lem}

\begin{proof}
If $I$ contains more than one point of the boundary $\bd(U)$ of a Fatou
domain $U$, then it contains an arc $K\subset \bd(U)$. As all Fatou
domains are (pre)periodic, $K$ maps eventually to a subarc $K'$ of
the boundary $\bd(V)$ of a periodic Fatou domain $V$. If $V$ is of degree greater
than $1$, then eventually $K'$ covers $\bd(V)$; since in this case $\bd(V)$
is a critical atom of $f|_X$, we are done. If $V$ is Siegel, then
eventually $K'$ covers all points of $\bd(V)$. Since we assume that $X$ is not
a Siegel atom, it follows that a critical point of $f|_X$ belongs to $\bd(V)$,
and again we are done. Thus, we may
assume that $I$ intersects boundaries of Fatou domains at most in single points.

By \cite{blolev02a} $I$ is not wandering;
we may assume that $I\cap
f(I)\ne \0$. Set $L=\bigcup^\infty_{k=0} f^k(I)$. Then $L$ is connected
and $f(L)\subset L$. If $\ol L$ contains a Jordan curve $Q$, then we
may assume that $Q$ is the boundary of an invariant Fatou domain.
If $L\cap Q=\0$, then $L$ is in a single component of $J\sm Q$,
hence $\ol L\cap Q$ is at most one point, a contradiction.
Choose the smallest $k$ with $f^k(I)\cap Q\ne \0$; by the assumption
$f^k(I)\cap Q=\{q\}$ is a singleton. Since the orbit of $I$ cannot
be contained in the union of components of $J\sm \{q\}$, disjoint from
$Q$, then there exists $m$ with $f^m(I)\cap Q$ not a singleton, a
contradiction.

Hence $\ol L$ is an invariant dendrite, and all cutpoints of $\ol L$
belong to images of $I$. Suppose that no cutpoint of $\ol L$ is
critical. Then $f|_{\ol L}$ is a homeomorphism (if two points of $\ol
L$ map to one point, there must exist a critical point in the open arc
connecting them). Now, it is proven in \cite{bfmot10} that if $D\subset
J$ is an invariant dendrite, then it contains infinitely many periodic
cutpoints. Hence we can choose two points $x, y\in \ol L$ and a number
$r$ such that $f^r(x)=x, f^r(y)=y$. Then the arc $I'\subset \ol L$
connecting $x$ and $y$ is invariant under the map $f^r$ which is
one-to-one on $I'$. Clearly, this is impossible inside $J$
(it is easy to see that a self-homeomorphism of an interval
with fixed endpoints and finitely many fixed points overall must have a fixed point
attracting from at least one side which is impossible in $J$).
\end{proof}

Let $\ell$ be a leaf of $\lam_\sim$. We equip $\lam_\sim$ with the
topology induced by the Hausdorff metric. Then $\lam_\sim$ is a
compact and metric space. Suppose that a leaf $\ell$ has a neighborhood
which contains at most countably many \emph{non-degenerate} leaves of $\lam_\sim$. Call such
leaves \emph{\textbf{c}ountably \textbf{i}solated} and denote the
family of all such leaves $\ci_\sim=\ci$. Clearly, $\ci$ is open in
$\lam_\sim$. Moreover,  $\ci$ is countable. To see this note that
$\ci$, being a subset of  $\lam_\sim$, is second countable and,
hence, Lindel\"of. Hence there exists a countable cover of $\ci$ all
of whose elements are countable and $\ci$ is countable as desired.
In terms of dynamics, $\ci$ 
is backward invariant and almost forward
invariant (except for critical leaves which may belong to $\ci$ and
whose images are points).

It can be shown that if we remove all leaves of $\ci$ from
$\lam_\sim$ (this is in the spirit of \emph{cleaning} of geometric
laminations \cite{thu85}), the remaining leaves form an invariant
geometric lamination $\lam^c_\sim$ (``c'' coming from ``countable
cleaning''). One way to see this is to use an alternative definition
given in \cite{bmov11}. A geo-lamination (initially not necessarily
invariant in the sense of Definition~\ref{geolaminv}) is called
\emph{sibling $d$-invariant lamination} or just \emph{sibling
lamination} if (a) for each $\ell\in\mathcal{L}$ either
$\sigma_d(\ell)\in\mathcal{L}$ or $\sigma_d(\ell)$ is a point in
$\uc$, (b) for each $\ell\in\mathcal{L}$ there exists a leaf
$\ell'\in\mathcal{L}$ with $\sigma_d(\ell')=\ell$, and (c) for each
$\ell\in\mathcal{L}$ with non-degenerate image $\sigma_d(\ell)$
there exist $d$ disjoint leaves $\ell_1, \dots, \ell_d$ in
$\mathcal{L}$ with $\ell=\ell_1$ and $\sigma_d(\ell_i) =
\sigma_d(\ell)$ for all $i$. By \cite[Theorem 3.2]{bmov11}, sibling
invariant laminations are invariant in the sense of
Definition~\ref{geolaminv}. Now, observe that $\lam_\sim$ is sibling
invariant. Since $\ci$ is open and contains the full grand orbit of
any leaf in it, $\lam^c_\sim$ is also sibling invariant. Thus, by
\cite[Theorem 3.2]{bmov11} $\lam^c_\sim$ is invariant in the sense
of Definition~\ref{geolaminv}. Infinite gaps of $\lam^c_\sim$ are
called \emph{super-gaps of $\sim$}.  Note that all finite gaps of
$\lam^c_\sim$ are also finite gaps of $\lam_\sim$.



Let $\lam^0_\sim=\lam_\sim$ and define $\lam^k_\sim$ inductively by
removing all isolated leaves from  $\lam^{k-1}_\sim$.

\begin{lem}\label{n=c}
There exists $n$ such that $\lam^n_\sim=\lam^c_\sim$; moreover, $\lam^c_\sim$
contains no isolated leaves.
\end{lem}

\begin{proof}We first show that there exists $n$ such that
$\lam^{n+1}_\sim=\lam^n_\sim$ (i.e., $\lam^n_\sim$ contains no
isolated leaves).  We note that increasing $i$ may only decrease the
number of infinite periodic gaps and decrease the number of finite
critical objects (gaps or leaves) in $\lam^i_\sim$. Therefore we may
choose $m$ so that $\lam^m_\sim$ has a minimal number of infinite
periodic gaps and finite critical objects. If $\lam^m_\sim$ has no
isolated leaf, then we may choose $n=m$. Otherwise, we will show
that we may choose $n=m+1$.

Suppose that $\lam^m_\sim$ has an isolated leaf $\ell$. Then $\ell$
is a common edge of two gaps $U$ and $V$. Since finite gaps of
$\lam_\sim$ (and, hence,  of all $\lam^i_\sim$) are disjoint we may
assume that $U$ is infinite. Moreover, it must be that $V$ is
finite. Indeed, suppose that $V$ is infinite and consider two cases.
First assume that there is a minimal $j$ such that $\si_d^j(U)=\si_d^j(V)$.
Then $\si_d^{j-1}(\ell)$ is critical and isolated in $\lam^m_\sim$.
Hence $\lam^{m+1}_\sim$ has fewer finite critical objects than
$\lam^m_\sim$, a contradiction with the choice of $m$. On the other
hand, if $U$ and $V$ never have the same image, then in
$\lam^{m+1}_\sim$ their periodic images will be joined. Then
$\lam^{m+1}_\sim$ would have a periodic gap containing both such
images which contradicts the choice of $m$ with the minimal number
of infinite periodic gaps. Thus, $V$ is finite.

Furthermore, the other edges of $V$ are not isolated in any lamination
$\lam^t_\sim, t\ge m$.  For if some
other edge $\bq$ of $V$ were isolated in $\lam^t_\sim$, then it would be an edge of
an infinite gap $H$, contained in the same gap of
$\lam^{t+1}_\sim$ as $U$. If $\si_d^r(H)=\si_d^r(U)$ for
some $r$, then $V$ is (pre)critical and the critical image of $V$ is
absent from $\lam^{t+1}_\sim$, a contradiction with the choice of
$m$. Otherwise, periodic images of $H$ and $U$ will be contained in
a bigger gap of $\lam^t_\sim$, decreasing the number of periodic
infinite gaps and again contradicting the choice of $m$. By definition this implies that
all edges of $V$ except for $\ell$ stay in all laminations $\lam^t_\sim, t\ge m$
and that $\lam^{m+1}_\sim$ has no isolated leaves.

Choose $n$ so that $\lam^n_\sim$ contains
no isolated leaves. It remains to be shown that
$\lam^n_\sim=\lam^c_\sim$.  Let $\ci''$ denote the union of all
leaves which are removed from $\lam_\sim$ to obtain $\lam^n_\sim$.
Clearly $\ci''\subset\ci$. Suppose that $N=\ci\sm\ci''\ne\0$. Note
that $\ci$ is an open subset of the compact metric space $\lam_\sim$
equipped with the Hausdorff topology. Hence $\ci$ is topologically
complete. Since $\ci''$ is countable, $N$ is a $G_\delta$ subset of
the complete space $\ci$ and, therefore, a Baire space. This
contradicts the fact that $\ci$ is countable.
\end{proof}

\begin{lem}\label{l:leafinsgap}
If $\sim$ be a lamination, then the following holds.

\begin{enumerate}

\item Every leaf of $\lam_\sim$ inside a super-gap $G$ of $\sim$ is
(pre)periodic or (pre)critical; every edge of a super-gap
is (pre)periodic.

\item Every edge of any gap $H$ of $\lam^c_\sim$ is  not
isolated in $\lam^c_\sim$ from outside of $H$; all gaps of
$\lam^c_\sim$ are pairwise disjoint. Moreover, gaps of
$\lam^c_\sim$ are disjoint from leaves which are not their edges.

\item There are no infinite concatenations of leaves in $\lam^c_\sim$.
Moreover, the geo-lamination $\lam^c_\sim$ gives rise to a
lamination $\sim^c$ such that the only difference between
$\lam^c_\sim$ and $\lam_{\sim^c}$ is as follows: it is possible that
one edge of certain finite gaps of $\sim^c$ is a leaf passing inside
an infinite gap of $\lam^c_\sim$.

\item All periodic Siegel gaps are properly contained in super-gaps.

\end{enumerate}

\end{lem}

\begin{proof}
(1) An isolated leaf $\ell$ of any lamination $\lam^k_\sim$ is either
(pre)periodic or (pre)critical. Indeed, since two finite gaps of
$\lam^k_\sim$ are not adjacent, $\ell$ is an edge of an infinite gap
$V$. Then by Lemma~\ref{l:maj}, $\ell$ is (pre)critical or
(pre)periodic. Since in the process of constructing $\lam^c_\sim$ we
remove isolated leaves of laminations $\lam^k_\sim$, we conclude
that all leaves inside a super-gap of $\sim$ (i.e., a gap of
$\lam^c_\sim$) are either (pre)periodic or (pre)critical.

Let $G$ be a gap of $\sim$.
Since by Lemma~\ref{n=c} $\lam^c_\sim$ contains no isolated leaves, every edge of
$G$ is a limit of leaves from outside of $G$. Hence
the only gaps of $\lam^c_\sim$ which can contain a critical leaf in their
boundaries are those which collapse to a single point (if a gap
$H$ of $\lam^c_\sim$ has a critical edge $\ell=\lim\ell_i$, then leaves
$\si_d(\ell_i)$ separate the point $\si_d(\ell)$ from the rest of the circle
which implies that $\si_d(H)=\si_d(\ell)$ is a point). Thus, super-gaps
have no critical edges, and by Lemma~\ref{l:maj} all their edges are
(pre)periodic.

(2) Every edge of any gap $H$ of $\lam^c_\sim$ is  not isolated in
$\lam^c_\sim$ from outside of $H$ by Lemma~\ref{n=c}. Moreover, no
point $a\in \uc$ can be an endpoint of more than two leaves of
$\lam_\sim$; hence, no point $a\in \uc$ can be an endpoint of more
than two leaves of $\lam^c_\sim$. We conclude that all gaps of
$\lam^c_\sim$ are pairwise disjoint and that all gaps of
$\lam^c_\sim$ are disjoint of all leaves which are not their edges.

(3) There are no infinite concatenations of leaves in $\lam^c_\sim$
because there are no such concatenations in $\lam_\sim$. Now
it is easy to see that the geometric lamination $\lam^c_\sim$ gives
rise to a lamination (equivalence relation) which we denote
$\sim^c$. Two points $a, b\in \uc$ are said to be
\emph{$\sim^c$-equivalent} if there exists a finite concatenation of leaves
of $\lam^c_\sim$ connecting $a$ and $b$. Since all leaves of $\lam^c_\sim$ are non-isolated,
it follows that the
geo-lamination $\lam^c_\sim$ and the geo-lamination
$\lam_{\sim^c}$ associated to $\sim^c$ can only differ as claimed.

(4) Clearly a Siegel gap $U$ is contained in a super-gap.  To see that it does not coincide
with a super-gap, observe that
there exists a $k$ such that $\si_d^k(U)$ has a critical edge.  Then
by (1), $\si_d^k(U)$ (and hence $U$ itself) is \emph{properly} contained in
a super-gap.
\end{proof}

Lemma~\ref{l:no1side} rules out certain dynamical behaviors of points.

\begin{lem}\label{l:no1side}
There exists no non-(pre)critical non-degenerate $\sim$-class $X$ such
that for every $n$, \emph{all} the sets $f^{n+k}(X), k>0$ are contained
in the same hole of $f^n(X)$.
\end{lem}

\begin{proof}
Suppose that such $\sim$-class $X$ exists.
Let us show that then the iterated images of $X$ cannot converge (along a subsequence of iterations) to a
critical leaf $\ell$.
Indeed, suppose that $\si^{n_k}_d(X)\to \ell$ so that $\si_d^{n_{k+1}}(X)$ separates $\si_d^{n_k}(X)$ from $\ell$.
Then $\si^{n_k+1}_d(X)\to \si_d(\ell)$, where $\si_d(\ell)$ is a point of $\uc$.
Since $\si_d^{n_{k+1}+1}(X)$ separates $\si_d(\ell)$ from $\si_d^{n_k+1}(X)$ for a sufficiently large $k$,
it follows by the assumption that the entire orbit of $\si^{n_k}(X)$ must be contained in a small component of
$\ol{\disk}\sm \si^{n_k}(X)$, containing $\si_d(\ell)$.
As this can be repeated for all sufficiently large $k$, we see that the limit set of $X$ has to coincide
with the point $\si_d(\ell)$,
a contradiction. Hence $X$ contains no critical leaves in its limit set.

Note that the assumptions of the lemma imply that $X$ is wandering.
By \cite{chi04} if $X$ is not a leaf then it contains a critical leaf
in its limit set. This implies that $X$ must be a leaf. For each
image $\si^n_d(X)$ let $Q_n$ be the component of $\ol{\disk}\sm \si^n_d(X)$ containing the rest of the orbit of $X$.
Let $W_n=\ol{\bigcap^n_{i=0} Q_n}$.
Then $W_n$ is a set whose boundary consists of finitely many leaves-images of $X$ alternating with
finitely many circle arcs.
On the next step, the image $\si^{n+1}_d(X)$ of $X$ is contained in $W_n$, and becomes a leaf on
the boundary of $W_{n+1}\subset W_n$.

Consider the set $W=\bigcap W_n$.
For an edge $\ell$ of $W$, let $H_W(\ell)$ be the hole of $W$ behind $\ell$.
When saying that a certain leaf is contained in $H_W(\ell)$, we mean
that its endpoints are contained in $H_W(\ell)$, or that the leaf is
contained in the convex hull of $H_W(\ell)$.
If $W$ is a point or a leaf, then the assumptions on the dynamics of
$X$ made in the lemma imply that $X$ converges to $W$ but never maps to $W$.
Clearly, this is impossible.
Thus, we may assume that $W$ is a non-degenerate convex
subset of $\ol{\disk}$ whose boundary consists of leaves and
possibly circle arcs.
The leaves in $\bd(W)$ can be of two types: limits of sequences of
images of $X$ (if $\ell$ is a leaf like that, then images of $X$
which converge to $\ell$ must be contained in $H_W(\ell)$), and images
of $X$ (if $\ell$ is such a leaf then there are only finitely many
images of $X$ in $H_W(\ell)$).
It follows that the limit leaves
from the above collection form the entire limit set of $X$;
moreover, by the above there are no critical leaves among them.

Let us show that this leads to a contradiction. First assume that
among boundary leaves of $W$ there is a limit leaf $\bq=\ol{xy}$ of
the orbit of $X$ (here $(x, y)=H_W(X)$ is the hole of $W$ behind $\bq$).
Let us show that $\bq$ is (pre)periodic. Indeed, since $\bq$
is approached from the outside of $W$ by images of $X$, and since
all images of $X$ are disjoint from $W$, it follows that $(\si_d(x),
\si_d(y))$ is the hole of $W$ behind $\si_d(X)$. Then by
Lemma~\ref{l:maj} and because there are critical leaves on the
boundary of $W$ we see that $\bq$ is (pre)periodic. Let $\ell$ is an
image of $\bq$ which is periodic. As $\ell$ is approached from
within $H_W(\ell)$ by images of $X$ and because $\ell$ is a repelling
leaf, we see that images of $X$ approaching $\ell$ from within
$H_W(\ell)$, are repelled farther away from $\ell$ inside $H_W(\ell)$.
Clearly, this contradicts the properties of $X$.

Now assume that there are no boundary leaves of $W$ which are limits
of images of $X=\ol{uv}$. Then all boundary leaves of $W$ are images
of $X$ themselves. Let us show that then there exists $N$ such that
for any $i\ge N$ we have that if the hole $H_W(\ell)$ of $W$ behind
$\ell=\si_d^i(X)$ is $(s, t)$ then $H_W(\si_d(\ell)=(\si_d(s),
\si_d(t))$. Indeed, first we show that if $H_W(\ell)=(s, t)$ while
$H_W(\si_d(\ell)=(\si_d(t), \si_d(s))$, then $(s, d)$ contains a
$\si_d$-fixed point. To see that, observe that in that case
$\si_d$-image of $[s, t]$ contains $[s, t]$ and the images of $s, t$
do not belong to $(s, t)$. This implies that there exists a
$\si_d$-fixed point in $[s, t]$. Since there are finitely many
$\si_d$-fixed points, it is easy to see that the desired number $N$
exists. Now we can apply Lemma~\ref{l:maj} which implies that $X$ is
either (pre)periodic or (pre)critical, a contradiction.
\end{proof}

\begin{lem}
\label{l:corept-in-sgap}
If $X$ is a
persistent cut-atom of $J$ such that $p^{-1}(X)$
is a subset of some super-gap of $\sim$, then either $X$ is
the boundary of a Siegel domain, or $X$ is a (pre)periodic point
which eventually maps to a periodic cutpoint. In any case,
$X$ eventually maps to $\pc$.
\end{lem}

\begin{proof}
We may assume that $X=x$ is a persistent cutpoint. Then the
$\sim$-class $p^{-1}(x)$ is non-trivial. If the boundary of this
$\sim$-class consists of (pre)critical leaves only, then the entire
class gets eventually collapsed, which is a contradiction with
$f^n(x)$ being cut-atoms for all $n\ge 0$. Therefore, there is a
leaf $\ell$ on the boundary of $p^{-1}(x)$ that is not
(pre)critical. Then this leaf is (pre)periodic by Lemma
\ref{l:leafinsgap}, hence $p^{-1}(x)$ is also (pre)periodic, and
eventually maps to a periodic gap or leaf.
\end{proof}

Given any subcontinuum $X\subset J$ and $x\in X$, a component of
$X\sm\{x\}$ is called an \emph{$X$-leg of x}.  An $X$-leg of $x$ is
called \emph{essential} if $x$ eventually maps into this leg. An
$X$-leg is said to be \emph{critical} if it contains at least one
critical atom; otherwise a leg is called {\em non-critical}.

Recall that by Definition~\ref{d:relam} we call a periodic atom
\emph{rotational} if it is of degree $1$ and its rotation number is
not zero. Denote the union of all periodic super-gaps by $SG$ (note
that there are only finitely many such super-gaps). Then $\pc_{rot}\sm p(SG)$
is the set of all periodic rotational atoms $x$ which are not
contained in $p(SG)$ (any such $x$ is a point by
Lemma~\ref{l:leafinsgap}).
Finally, define $\ce_{s}$ as the dynamical span of the limit sets of all persistent cut-atoms $x$
(equivalently, cutpoints) which never map into $p(SG)$. Observe that
if $J$ is a dendrite, then $\ce_s=\ce$.

Now we study dynamics of super-gaps and the map as a whole. Our
standing assumption from here through Theorem~\ref{t:corexpli} is
that $\sim$ is a lamination and $\lam_\sim$ is such that
$\lam^c_\sim$ is not empty (equivalently, $\uc$ is not a super-gap).
Then there are finitely many super-gaps in $SG$, none of which
coincide with $\uc$, and by Lemma~\ref{l:leafinsgap} they are
disjoint. Choose $N_\sim=N$ as the minimal such number that all
periodic super-gaps and their periodic edges are $\si_d^N$-fixed.
Set $g_\sim=g=f^N$. By Lemma~\ref{l:leafinsgap} each super-gap has
at least one $g$-fixed edge and all its edges eventually map to
$g$-fixed edges.

Consider a component $A$ of $J\sm p(SG)$. There are several
$\si_d^N$-fixed super-gaps bordering $p^{-1}(A)$, and each such
super-gap has a unique well-defined edge contained in
$\ol{p^{-1}(A)}$. If all these edges are $\si_d^N$-fixed, we call
$A$ \emph{settled}. By Theorem~\ref{t:fxpt}, a settled $A$ contains
an element of $\pc_{rot}\sm p(SG)$ denoted by $y_A$. In this way, we
associate elements of $\pc_{rot}\sm p(SG)$ to all settled components
of $J\sm p(SG)$.

\begin{lem}\label{l:settle}
If $\ell$ is not a $\si^N_d$-fixed edge of a $\si_d^N$-fixed
super-gap $H$, then the component of $J\sm p(\ell)$ which contains
$p(H)$, contains a settled component of $J\sm p(SG)$. In particular,
settled components exist.
\end{lem}

\begin{proof}
Choose a $\si^N_d$-fixed edge $\ell'$ of $H$ and a component $B$ of
$J\sm p(SG)$ such that $\ell'$ is contained in the closure of
$p^{-1}(B)$. If $B$ is settled we are done. Otherwise find a
super-gap $G''$ with an edge $\ell''$ which ($\ell''$) is not
$\si_d^N$-fixed and borders $p^{-1}(B)$, then proceed with $\ell''$
as before with $\ell$.

In the end we will find a settled component of $J\sm
p(SG)$ in the component of $J\sm p(\ell)$ containing $p(H)$. Indeed,
on each step we find a new $\si_d^N$-fixed super-gap different from
the preceding one. Since there are finitely many $\si_d^N$-fixed
super-gaps, we either stop at some point, or form a cycle. The
latter is clearly impossible. Thus, there exists a non-empty
collection of settled components $A$.
\end{proof}

We are ready to prove the main result dealing with super-gaps.

\begin{thm}\label{t:corexpli}
If $x$ is a persistent cutpoint never mapped to $p(SG)$ then there
is $n\ge 0$ such that $f^n(x)$ separates two points of $\pc_{rot}\sm
p(SG)$ and is a cutpoint of $\ic(\pc_{rot}\sm p(SG))$ so that
$\ce_s=\ic(\pc_{rot}\sm p(SG))$.

Moreover, there exist infinitely many persistent periodic rotational
cutpoints outside $p(SG)$, $\ce_s\subset \ce_{rot}$, and any periodic
cutpoint outside $p(SG)$ separates two points of $\pc_{rot}\sm p(SG)$
and is a cutpoint of\, $\ce_s$, of\, $\ce_1$ and of\, $\ce$.
\end{thm}

\begin{proof}
We first prove that \emph{every $x$-essential leg of every point $x\in J$ contains a
point of $\pc_{rot}\sm p(SG)$}.
This will imply that for a non-(pre)periodic persistent cutpoint $x$
there exists $n$ such that $f^n(x)$ separates two points of $\pc_{rot}\sm p(SG)$
because, by Lemma~\ref{l:no1side},
some iterated $g$-image of $x$ has at least two $x$-essential $J$-legs.

Let $L$ be an $x$-essential $J$-leg of $x$. Denote by $A$ the
component of $L\sm p(SG)$ which contains $x$ in its closure. If $L$
contains no $p$-images of $\si^N_d$-fixed super-gaps (which implies
that $L=A$), then the claim follows from Theorem~\ref{t:fxpt}
applied to $\ol{A}$. Otherwise there are finitely many super-gaps
$U_1, \dots, U_t$ such that $p(U_i)$ borders $A$. For each $i$ let
us take a point $x'_i\in \bd(p(U_i))$ that separates $x$ from the
rest of $p(U_i)$. If all the points $x'_i$ are $g$-fixed, then we
are done by Theorem~\ref{t:fxpt} applied to $\ol{A}$. If there
exists $i$ such that the point $x'_i$ is not $g$-fixed, then, by
Lemma~\ref{l:settle}, $x'_i$ separates the point $x$ from some
settled component $B$, which in turn contains an element
$y_B\in\pc_{rot}\sm p(SG)$. Thus, in any case \emph{every
$x$-essential leg of $x\in J$ contains a point of $\pc_{rot}\sm
p(SG)$} and the claim from the first paragraph is proven.

Now consider the case of a $g$-periodic cutpoint $y$ of $J$ outside
$p(SG)$ (a priori it may happen that $y$ above is an endpoint of
$\ce_s$). We want to show that \emph{$y$ separates two points of
$\pc_{rot}\sm p(SG)$}. Indeed, a certain power $(\si^N_d)^k$ of
$\si^N_d$ fixes $p^{-1}(y)$ and has rotation number zero on
$p^{-1}(y)$. 
Choose an edge $\ell$ of $p^{-1}(y)$ and consider the component $B$
of $J\sm \{y\}$ such that $\ol{p^{-1}(B)}$ contains $\ell$.

Then $(\si^N_d)^k$ fixes $\ell$ while
leaves and gaps in $\ol{p^{-1}(B)}$, close to $\ell$, are repelled away from
$\ell$ inside $\ol{p^{-1}(B)}$ by $(\si^N_d)^k$. Hence their $p$-images are repelled
away from $y$ inside $B$ by the map $g^k$. By the previous paragraph this implies
that there is an element (a point) $t_B\in\pc_{rot}\sm p(SG)$ in $B$. As this
applies to all edges of $p^{-1}(y)$, we see that \emph{$y$ separates two
points of $\pc_{rot}\sm p(SG)$}. As $\pc_{rot}\sm p(SG)\subset \ce_s\subset \ce$, this
proves that \emph{any periodic cutpoint outside $p(SG)$ separates two points
of $\pc_{rot}\sm p(SG)$ and is a cutpoint of\,
$\ce_s$ (and, hence, of\, $\ce_1$ and of\, $\ce$)}.
This also proves that \emph{for a (pre)periodic persistent cutpoint $x$
there exists $n$ such that $f^n(x)$ separates two points of $\pc_{rot}\sm p(SG)$};
by the above, it suffices to take $r$ such that $g^r(x)=f^{Nr}$ is periodic and
set $n=Nr$.

Let us prove that \emph{there are infinitely many points of
$\pc_{rot}$ in any settled component $A$}. Indeed, as above choose a
$g$-periodic point $y\in A$ such that a certain power $(\si^N_d)^k$
of $\si^N_d$ which fixes $p^{-1}(y)$ has rotation number zero on
$p^{-1}(y)$. Choose an edge $\ell$ of $p^{-1}(y)$ and consider the
component $B$ of $A\sm \{y\}$ such that $\ol{p^{-1}(B)}$ contains
$\ell$. Then leaves and gaps in $\ol{p^{-1}(B)}$, which are close to
$\ell$, are repelled away from $\ell$ inside $\ol{p^{-1}(B)}$ by
$(\si^N_d)^k$. Hence their $p$-images are repelled away from $y$
inside $B$ by $g^k$. By Theorem~\ref{t:fxpt} this implies that there
exists a $g^k$-fixed point $z\in B$ with non-zero rotation number.
Replacing $A$ by a component of $A\sm z$, we can repeat the same
argument. If we do it infinitely many times, we will prove that
\emph{there are infinitely many points of $\pc_{rot}$ in any settled
component $A$}.
\end{proof}

Corollary~\ref{c:corexpli2} immediately follows from
Theorem~\ref{t:corexpli}. Observe that if $J$ is a dendrite then
$SG=p(SG)=\0$ and hence $\ce=\ce_1=\ce_s$.

\begin{cor}\label{c:corexpli2}
If $J$ is a dendrite then $\ce=\ce_1=\ce_s=\ce_{rot}=\ic(\pc_{rot})$.
Furthermore, for any persistent cutpoint $x$ there is $n\ge 0$ such
that $f^n(x)$ separates two points from $\pc_{rot}$ (thus, at some
point $x$ maps to a cutpoint of $\ce$). Moreover, any periodic cutpoint
separates two points of $\pc_{rot}$ and therefore is itself a cutpoint
of $\ce$.
\end{cor}

\begin{proof}
Left to the reader.
\end{proof}

Corollary~\ref{c:corexpli2} implies the last, dendritic part of
Theorem~\ref{t:criticoreintr}. The rest of
Theorem~\ref{t:criticoreintr} is proven below.

\smallskip

\noindent \emph{Proof of Theorem~\ref{t:criticoreintr}}. By definition,
$\ic(\pc)\subset\ce, \ic(\pc_1)\subset \ce_1$ and
$\ic(\pc_{rot})\subset \ce_{rot}$. To prove the opposite inclusions,
observe that by definition in each of these three cases it suffices to
consider a persistent cut-atom $X$ which is not (pre)periodic. However
in this case by Theorem~\ref{t:corexpli} there exists $n$ such that
$f^n(X)$ separates two points of $\pc_{rot}\sm p(SG)$ and therefore is
contained in
$$\ic(\pc_{rot}\sm p(SG))\subset \ic(\pc_{rot})\subset \ic(\pc_1)\subset \ic(\pc)$$
which proves all three inclusions of the theorem. \hfill \qed

Theorems~\ref{t:criticoreintr} and ~\ref{t:corexpli} give explicit
formulas for various versions of the dynamical core of a topological
polynomial $f$ in terms of various sets of periodic cut-atoms. These
sets of periodic cut-atoms are most likely infinite. It may
also be useful to relate the sets $\ce$ or $\ce_s$ to a finite set
of critical atoms. To do this, we need new notation. Let $\crA(X)$
be the set of critical atoms of $f|_X$; set $\crA=\crA(J)$.
Also, denote by $\pc(X)$ the set of all periodic cut-atoms of $X$
and by $\pc_1(X)$ the set of all periodic cut-atoms of $X$ of degree
$1$.


\begin{lem}
\label{l:cutgorel} Let $X\subset J$ be an invariant complete
continuum which is not a Siegel atom. Then the following facts hold.

\begin{enumerate}

\item For every cutpoint $x$ of $X$, there exists an integer
$r\ge 0$ such that $f^r(x)$ either {\rm(a)} belongs to a set from
$\crA(X)$, or {\rm(b)} separates two sets of $\crA(X)$, or {\rm(c)}
separates a set of $\crA(X)$ from its image. In any case $f^r(x)\in
\ic(\crA(X))$, and in cases {\rm(a)} and {\rm(b)} $f^r(x)$ is a
cutpoint of $\ic(\crA(X))$. In particular, the dynamical span of all
cut-atoms of $X$ is contained in $\ic(\crA(X))$.

\item $\pc(X)\subset \ic(\crA(X))$. In particular, if
$X=\ic(\pc(X))$ then $X=\ic(\crA(X))$.


\end{enumerate}

\end{lem}

\begin{proof}
(1) Suppose that $x$ does not eventually map to $\crA(X)$.
Then all points in the forward orbit of $x$ are cutpoints of $X$
(in particular, there are at least two $X$-legs at any such point).

If $f^{r}(x)$ has more than one critical $X$-leg for some $r\ge 0$,
then $f^{r}(x)$ separates two sets of $\crA(X)$. Assume that
$f^k(x)$ has one critical $X$-leg for every $k\ge 0$. By
Lemma~\ref{l:easytop}, each non-critical $X$-leg $L$ of $f^{k}(x)$
maps in a one-to-one fashion to some $X$-leg $M$ of
$f^{k+1}(x)$. There is a connected neighborhood $U_k$ of $f^k(x)$ in
$X$ so that $f|_{U_k}$ is one-to-one. We may assume that $U_k$
contains all of the non-critical legs at $f^k(x)$.  Hence there
exists a bijection $\ph_k$ between components of $U_k\sm f^k(x)$ and
components of $f(U_k)\sm f^{k+1}(x)$ showing how small pieces
(\emph{germs}) of components of $X\sm \{f^k(x)\}$, containing
$f^k(x)$ in their closures, map to small pieces (\emph{germs}) of
components of $X\sm \{f^{k+1}(x)\}$, containing $f^{k+1}(x)$ in their
closures.

By Lemma~\ref{l:condens} choose $r>0$ so that a non-critical $X$-leg
of $f^{r-1}(x)$ maps to the critical $X$-leg of $f^{r}(x)$. Then the
bijection $\ph_{r-1}$ sends the germ of the critical $X$-leg $A$ of
$f^{r-1}(x)$ to a non-critical $X$-leg $B$ of $f^{r}(x)$. Let
$C\subset A$ be the connected component of
$X\setminus(\crA(X)\cup\{f^{r-1}(x)\})$ with $f^{r-1}(x)\in \ol{C}$;
let $R$ be the union of $\ol{C}$ and all the sets from $\crA(X)$
non-disjoint from $\ol{C}$. Then $f(R)\subset \ol{B}$ while all the
critical atoms are contained in the critical leg $D\ne B$ of
$f^{r}(x)$. It follows that $f^{r}(x)$ separates these critical
atoms from their images. By definition of $\ic(\crA(X))$ this
implies that either $f^r(x)$ belongs to a set from $\crA(X)$, or
$f^r(x)$ is a cutpoint of $\ic(\crA(X))$. In either case $f^r(x)\in
\ic(\crA(X))$. The rest of (1) easily follows.

(2) If $x$ is a periodic cutpoint of $X$ then, choosing $r$ as
above, we see that $f^r(x)\in \ic(\crA(X))$ that implies that $x\in
\ic(\crA(X))$ (because $x$ is an iterated image of $f^r(x)$ and
$\ic(\crA(X))$ is invariant). Thus, all periodic cutpoints of $X$
belong to $\ic(\crA(X))$. Now, take a periodic Fatou atom $Y$. If
$Y$ is of degree greater than $1$, then it has an image $f^k(Y)$
which is a critical atom of $f|_X$. Thus, $Y\subset \ic(\crA(X))$.
Otherwise for some $k$ the set $f^k(Y)$ is a periodic Siegel atom
with critical points on its boundary. Since $X$ is not a Siegel atom
itself, $f^n(Y)$ contains a critical point of $f|_X$. Hence the
entire $Y$ is contained in the limit set of this critical point and
again $Y\subset \ic(\crA(X))$. Hence each periodic Fatou atom in $X$
is contained in $\ic(\crA(X))$. Thus, $\pc(X)\subset \ic(\crA(X))$
as desired.
\end{proof}

We can now relate various dynamical cores to the critical
atoms contained in these cores. However first let us consider the
following heuristic example. Suppose that the lamination $\sim$ of
sufficiently high degree has an invariant Fatou gap $V$ of degree
$2$ and, disjoint from it,  a super-gap $U$ of degree $3$. The
super-gap $U$ is subdivided (``tuned'') by an invariant quadratic
gap $W\subset U$ with a critical leaf on its boundary so that $W$
concatenated with its appropriate pullbacks fills up $U$ from
within.

Also assume that the strip between $U$ and $V$ is enclosed by two
circle arcs and two edges, a fixed edge $\ell_u$ of $U$ and a
prefixed edge $\ell_v$ of $V$ (that is, $\si_d(\ell_v)$ is a fixed
edge of $V$). Moreover, suppose that  $U$ and $V$ have  only two
periodic edges , namely, $\ell_u$ and $\si_d(\ell_v)$, so that all
other edges of $U$ and $V$ are preimages of $\ell_u$ and
$\si_d(\ell_v)$. All other periodic gaps and leaves of $\sim$ are
located in the component $A$ of $\disk\sm \si_d(\ell_v)$ which does
not contain $U$ and $V$.

It follows from Theorem~\ref{t:criticoreintr} that in this case
$\ce_1$ includes a continuum $K\subset p(A\cup V)$ united with a
connector-continuum $L$ connecting $K$ and $p(\ell_u)$. Moreover,
$p(\bd(V))\subset K$. Basically, all the points of $L$ except for
$p(\ell_u)$ are ``sucked into'' $K$ while being repelled away from
$p(\ell_u)$. Clearly, in this case even though
$p(\ell_u)\in \ce_1$, still $p(\ell_u)$ does not belong to the set
$\ic(\crA(\ce_1))$ because $p(\bd(W))$, while being a critical atom, is
not contained in $\ce_1$. This shows that some points of $\ce_1$ maybe
located outside $\ic(\crA(\ce_1))$ and also that there might exist
non-degenerate critical atoms intersecting $\ce_1$ over a point (and
hence not contained in $\ce_1$). This example shows that the last claim of
Lemma~\ref{l:cutgorel} cannot be established for $X=\ce_1$.

Also, let us consider the case when $\ce$ is a Siegel atom $Z$. Then
by definition there are no critical points or atoms of $f|_\ce$,
so in this case $\crA(\ce)$ is empty. However this is the only exception.

\begin{thm}
\label{t:critcore}
Suppose that $\ce$ is not a unique Siegel atom. Then
$\ce=\ic(\crA(\ce))$,
$\ce_{rot}=\ic(\crA(\ce_{rot}))$, $\ce_{s}=\ic(\crA(\ce_{s}))$.
\end{thm}

\begin{proof}
By Theorem~\ref{t:criticoreintr} $\ce=\ic(\pc)$. Let us show that in
fact $\ce$ is the dynamical span of its periodic cutpoints and its
periodic Fatou atoms. It suffices to show that any periodic cutpoint
of $J$ either belongs to a Fatou atom or is a cutpoint of $\ce$.
Indeed, suppose that $x$ is a periodic cutpoint of $J$ which does
not belong to a Fatou atom. Then by Theorem~\ref{t:fxpt} applied to
different components of $J\sm \{x\}$ we see that $x$ separates two
periodic elements of $\pc$. Hence $x$ is a cutpoint of $\ce$ as
desired. By Lemma~\ref{l:cutgorel} $\ce=\ic(\crA(\ce))$. The
remaining claims can be proven similarly.
\end{proof}

There is a bit more universal way of stating a
similar result. Namely, instead of considering critical atoms of
$f|_\ce$ we can consider critical atoms of $f$ \emph{contained in $\ce$}.
Then the appropriately modified claim of Theorem~\ref{t:critcore}
holds without exception. Indeed, it holds trivially in the case when
$\ce$ is an invariant Siegel atom. Otherwise it follows from
Theorem~\ref{t:critcore} and the fact that the family of critical
atoms of $f|_\ce$ is a subset of the family of all critical atoms of
$f$ contained in $\ce$. We prefer the statement of
Theorem~\ref{t:critcore} to a more universal one because it allows
us not to include ``unnecessary'' critical points of $f$ which
happen to be endpoints of $\ce$; clearly, the results of
Theorem~\ref{t:critcore} hold without such critical points. Notice,
that the explanations given in this paragraph equally relate to
$\ce_1$ and $\ce_{rot}$.

In the dendritic case the following corollary holds.

\begin{cor}\label{t:critcoredendr}
If $J$ is a dendrite, then the following holds:
$$\ce=\ic(\pc_{rot})=\ic(\crA(\ce))$$.
\end{cor}

\begin{proof}
Left to the reader.
\end{proof}

\section{Invariant quadratic gaps and their canonical laminations}\label{s:qgaps}

In this section, we assume that $U$ is a $\si_3$-invariant quadratic gap
and study its properties.
We then define \emph{canonical}
laminations, which correspond to these gaps and describe other
laminations which \emph{refine} the canonical ones. To lighten the
notation, we will write $\si$ instead of $\si_3$ throughout the
rest of the paper.

\subsection{Invariant quadratic gaps}\label{s:invquagap}

\begin{lem}\label{bndcrit}
Let $U$ be an invariant quadratic gap. Then there exists a unique
edge $\ell$ of $U$ with $|\ell|_U\ge\frac13$. Moreover,
$|\ell|_U<\frac 23$, on all holes $H''\ne H(\ell)$ of $U$ the map
$\si$ is a homeomorphism onto its image, and $|\si(H'')|=3|H''|$.
The following cases are possible:

\begin{enumerate}

\item $|\ell|_U=\frac13$.

\item $\ell$ is periodic of some period $k$, $|\si(\ell)|_U=\frac1{3^k-1}$
and $|\ell|_U=3^{k-1}\cdot\frac 1{3^k-1}$.


\end{enumerate}

\end{lem}

\begin{proof}
Take an edge $\ell$ of $U$ with $|\ell|_U$ maximal. If
$|\ell|_U<1/3$, then $|\si(\ell)|_U=3 |\ell|_U$, a contradiction.
Hence, $|\ell|_U\ge 1/3$.
Observe that if a set $A\subset\uc$ lies in the
complement to two closed arcs of length $\ge 1/3$ each, then the
restriction of $\si$ to $A$ is injective. This implies that all
holes $H''\ne H(\ell)$ of $U$ are shorter than $\frac13$ and that
$|\ell|_U<\frac23$ (here we use the fact that $\si|_{U'}$ is
two-to-one).

Clearly, $|\ell|_U$ can be equal to $\frac13$ (just take
$\ell=\ol{\frac13 \frac23}$ and assume that $U$ is the convex hull of
the set of all points $x\in \uc$ with orbits outside the arc $(\frac13,
\frac23)$). This situation corresponds to case (1) of the lemma.

Suppose  $|\ell|_U>\frac13$ but $\ell$ is not periodic. Since there
are $2^n$ components of $\si^{-n}(H(\ell))$ that are holes of $U$,
and the length of each such component is $|\ell|_U/3^n$, the sum of
lengths of all such components equals
$|H(\ell)|(1+\frac23+\dots+\frac{2^n}{3^n}+\dots)=3|H(\ell)|>1,$ a
contradiction. Thus, $\ell$ is periodic.

Finally, let $k$ be the period of $\ell$.
We show that $|\ell|_U=\frac{3^{k-1}}{3^k-1}$.
Let $|\si(\ell)|_U=x$.
Then $|\si^{i+1}(\ell)|=3^ix$ for $i=1,\dots,k-1$, and hence $|H(\ell)|=3^{k-1}x$.
On the other hand,
$|\si(\ell)|_U=3|\ell|_U-1$ (recall that $|\ell|_U<\frac23$).
Thus, $x=3^kx-1$, $x=\frac1{3^k-1}$ and
$|\ell|_U=3^{k-1}\cdot\frac 1{3^k-1}$.
This implies that
$|\si(\ell)|_U=(3|\ell|_U-1)=\frac1{3^k-1}$.
\end{proof}

By Definition~\ref{d:0major}, the unique leaf $\ell$ whose existence was proven in
Lemma~\ref{bndcrit} is the unique \emph{major (leaf)} of an
invariant quadratic gap $U$. It will be from now on denoted by
$M(U)$. Lemma~\ref{bndcrit} implies a simple description of the
basis $U'$ of $U$ if the major $M(U)$ is given.

\begin{lem}\label{descru}
The basis $U'$ of $U$ is the set of points $x$ whose
orbits are disjoint from $H(M(U))$. All edges of $U$ are preimages of $M(U)$.
\end{lem}

\begin{proof}
All points of $U'$ have orbits disjoint from $H(M(U))$. On the other
hand, if $x\in \uc\sm U'$, then $x$ lies in a hole of $U$ behind a leaf
$\ell$. By Lemma~\ref{l:maj}, the orbit of $\ell$ contains $M(U)$. Hence for
some $n$ we have $\si^n(\ell)=M(U)$, $\si^n(H(\ell))=H(M(U))$, and
$\si^n(x)\in H(M(U))$.
\end{proof}

Thus, invariant quadratic gaps $U$ are generated by their majors. It
is natural to consider the two cases from Lemma~\ref{bndcrit}
separately. We begin with the case when an invariant quadratic gap
$U$ has a periodic major $M(U)=\ol{ab}$ of period $k$ and
$H(M(U))=(a, b)$. Set $a''=a+\frac13$ and $b''=b-\frac13$. Set
$M''(U)=\ol{b''a''}$. Consider the set $N(U)$ of all points with
$\si^k$-orbits contained in $[a, b'']\cup [a'', b]$ and its convex
hull $V(U)=\ch(N(U))$ \emph{(this notation is used in several lemmas
below)}.

\begin{lem}\label{vassal}
Assume that $M(U)$ is periodic. Then $N(U)$ is a Cantor set; $V(U)$
is a periodic stand alone quadratic gap of period $k$.
\end{lem}

\begin{proof}
Under the action of $\si$ on $[a, b]$ the arc $[a, a'']$ wraps around
the circle once, and the arc $[a'', b]$ maps onto the arc $[\si(a),
\si(b)]$ homeomorphically. Similarly, we can think of $\si|_{[a,
b]}$ as homeomorphically mapping $[a, b'']$ onto the arc $[\si(a),
\si(b)]$ and wrapping $[b'', b]$ around the circle once.
Thus, first the arcs $[a, b'']$ and $[a'', b]$ map homeomorphically to
the arc $[\si(a), \si(b)]$ (which is the closure of a hole of
$U$, see Lemma~\ref{bndcrit}). Then, under further iterations of
$\si$, the arc $[\si(a), \si(b)]$ maps homeomorphically onto
closures of distinct holes of $U$ until $\si^{k-1}$ sends
it, homeomorphically, onto $[a, b]$. This
generates the quadratic gap $V(U)$ contained in the strip between
$\ol{ab}$ and $\ol{a''b''}$. In the language of one-dimensional
dynamics one can say that closed intervals $[a, b'']$ and $[a'', b]$
form a 2-horseshoe of period $k$. A standard argument shows that
$N(U)$ is a Cantor set. The remaining claim easily follows.
\end{proof}


\begin{figure}
\includegraphics[width=6cm]{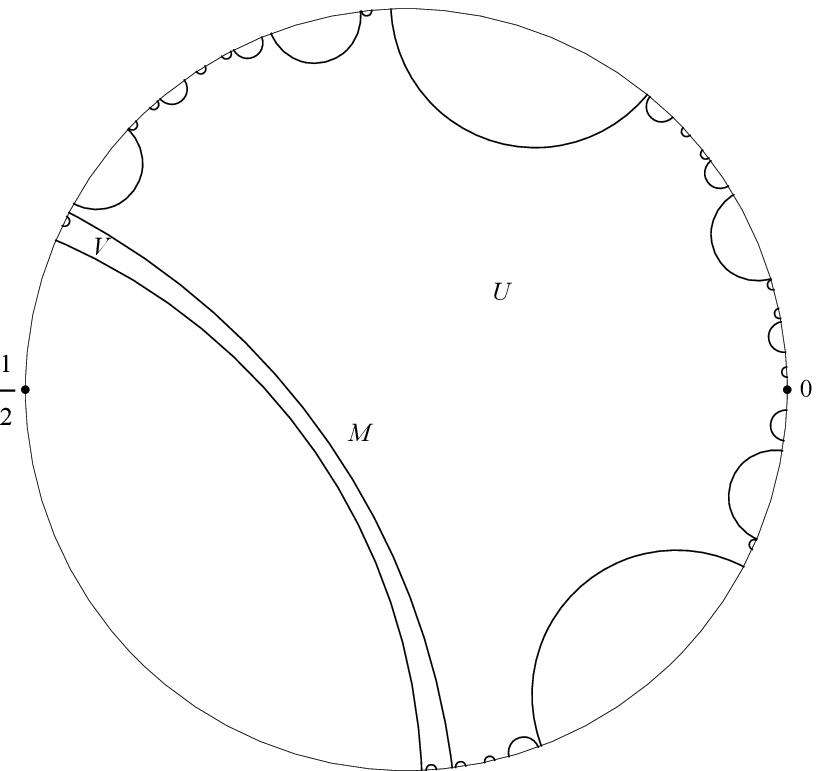}
\includegraphics[width=6cm]{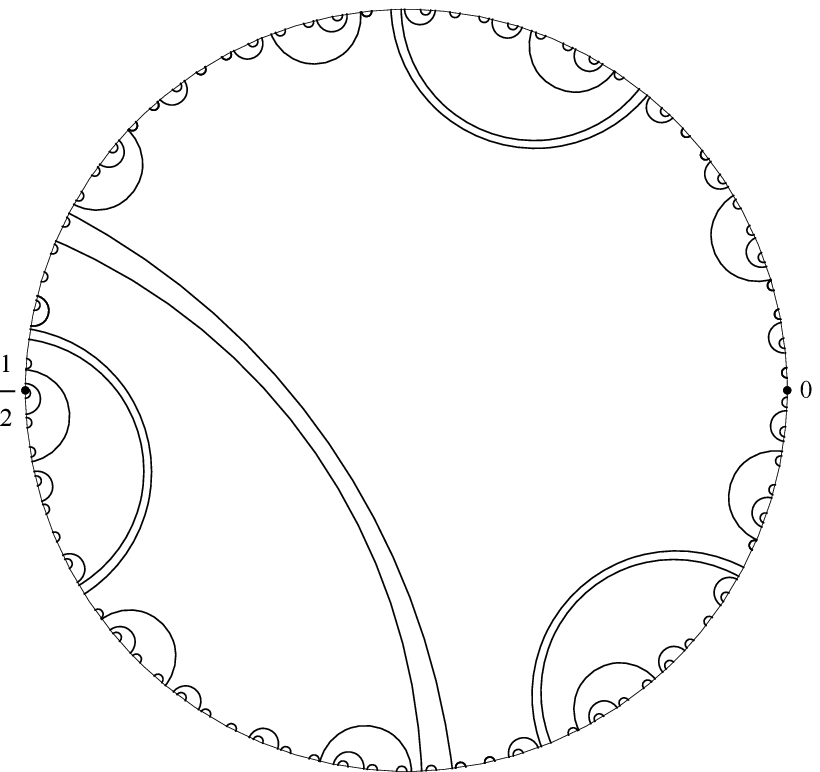}
{\caption{Left: a quadratic invariant gap $U$ of periodic type and
its vassal $V$. Right: its canonical lamination (we draw geodesics
in the Poincar\'e metric instead of chords to make the pictures look
better).}}
\end{figure}


We call $V(U)$ the \emph{vassal} (gap) of $U$. We also
define another type of gap called a \emph{caterpillar} (gap). Let
$Q$ be a periodic gap of period $k$. Suppose that $\bd(Q)$ consists
of a periodic leaf $\ell_0=\ol{xy}$ of period $k$, a critical leaf
$\ell_{-1}=\ol{yz}$ concatenated to it, and a countable chain
of leaves $\ell_{-n}$, concatenated to $\ell_{-1}$ and accumulating
at $x$ ($\ell_{-r-1}$ is concatenated to $\ell_{-r}, r=1, 2,
\dots$). Suppose that $\si^k(x)=x$, $\si^k(\{y, z\})=y$ and $\si^k$
maps each $\ell_{-r-1}$ to $\ell_{-r}$ (all leaves are shifted by
one towards $\ell_0$ except for $\ell_0$ which maps to itself and
$\ell_{-1}$ which collapses to the point $y$). The leaf $\ell_0$ is
called the \emph{head} of $Q$. Similar gaps are already useful for
quadratic laminations (see \cite{thu85}) where the invariant gap
with edges $\ol{0\frac 12}, \ol{\frac 12\frac 14}, \dots, \ol{\frac
1{2^n}\frac 1{2^{n+1}}}, \dots$ is considered). Lemma~\ref{catpil}
gives examples of caterpillar gaps and is left to the reader.


\begin{lem}\label{catpil}
Suppose that $M(U)=\ol{ab}$ is periodic of period $k$. Then one can construct
two \ca\, gaps with head $M(U)$ such that their bases are contained
in $\ol{H(M(U))}$. In the first of them, the critical leaf is
$\ol{aa'}$, and in the second one the critical leaf is $\ol{bb'}$.
\end{lem}

We call the \ca\, gaps from Lemma~\ref{catpil} \emph{canonical \ca\,
gaps of $U$}. A critical edge $c$ of a canonical \ca\, gap defines
it, so this caterpillar gap is denoted by $C(c)$. We denote its
basis by $C'(c)$. To study related invariant quadratic gaps we first
prove the following general lemma in which we adopt a different
point of view. Namely, any leaf $\ell$ which is not a diameter
defines an open arc $L(\ell)$ of length greater than $\frac12$ (in
particular, a critical leaf $c$ defines an arc $L(c)$ of length
$\frac23$). Let $\Pi(\ell)$ be the set of all points with orbits in
$\ol{L(\ell)}$.

\begin{lem}\label{pic}
Suppose that $c$ is a critical leaf. The set $\Pi(c)$ is non-empty,
closed and forward invariant. A point $x\in \Pi(c)$ has two
preimages in $\Pi(c)$ if $x\ne \si(c)$, and three preimages in
$\Pi(c)$ if $x=\si(c)$.
The convex hull $G(c)$ of $\Pi(c)$ is a stand alone invariant quadratic gap.
If $x=\si(c)\in \Pi(c)$, we have the
situation corresponding to case {\rm(1)} of Lemma~\ref{bndcrit}.
\end{lem}

\begin{proof}
It is easy to see that the set $\Pi(c)$ is closed and forward invariant;
it is non-empty because it contains at least one fixed point. Let $x\in
\Pi(c)$. If $x\ne \si(c)$, then of its three preimages one belongs to
$\uc\sm \ol{L(c)}$ while two others are in $\ol{L(c)}$, and hence, by definition,
also in $\Pi(c)$. Suppose that $x=\si(c)$ (and so, since by the
assumption $x\in \Pi(c)$, the orbit of $c$ is contained in $\ol{L(c)}$).
Then the entire triple $\si^{-1}(\si(c))$ is contained in $\ol{L(c)}$ and,
again by definition, in $\Pi(c)$.

To prove the last claim of the lemma, we prove that \emph{any hole
$I$ of $\Pi(c)$ except for the hole $T$, whose closure contains the
endpoints of $c$, maps to a hole of $\Pi(c)$}. Indeed, otherwise
there is a point $y\in I$ such that $\si(y)$ is a point of $\Pi(c)$.
Since $I\subset \ol{L(c)}$, then $y\in \Pi(c)$, a contradiction.
Consider the hole $T=(a,b)$ defined above. If $T=\uc\sm \ol{L(c)}$,
there is nothing to prove as in this case the leaf $c$ is a critical
edge of $G(c)$ that maps to a point of $\Pi(c)$. Suppose now that
$T\ne \uc\sm \ol{L(c)}$ and that there is a point $z\in (\si(a),
\si(b))\cap \Pi(c)$. Then there is a point $t\in T\cap L(c)$ with
$\si(t)=z$ and hence $t\in \Pi(c)$, a contradiction. Thus, $(\si(a),
\si(b))$ is a hole of $\Pi(c)$, and $\ol{\si(a)\si(b)}$ is an edge
of $G(c)$. All this implies easily that $G(c)$ is an invariant gap,
and it follows from the definition that it is quadratic. The last
claim easily follows.
\end{proof}

Below we will use the notation $G(c)=\ch(\Pi(c))$. Let us relate
invariant quadratic gaps defined in terms of periodic majors (see
Lemma~\ref{vassal}) and caterpillar gaps $C(c)$ (see
Lemma~\ref{catpil}) to gaps $G(c)$.

\begin{lem}\label{picper}
Let $M(U)=\ol{ab}$ be periodic of period $k$ and $c$ be a critical
leaf. Then {\rm(1)} if the endpoints of $c$ belong to $H(M(U))$ then
$U'=\Pi(c)$, {\rm(2)} if one endpoint of $c$ is $a$ or $b$, then
$\Pi(c)$ consists of $U'$, $C'(c)$, and all (iterated) pullbacks of
$C'(c)$ to holes of $U$.
\end{lem}

\begin{proof}
(1) Clearly $U'\subset \Pi(c)$. Let $\ell=\ol{b'a'}$; then $\ell$ is
an edge of the vassal gap $V(U)$ and $\ell\cap c=\0$. Hence the
closure of the arc $I=H_{V(U)}(\ell)$ consists of points not in
$\Pi(c)$. Clearly, other points of $H(M(U))\sm N(U)$ map eventually
to $\ol{I}$. Thus, they do not belong to $\Pi(c)$ either. If the
orbit of a point from $N(U)$ is not separated by $c$, then this
point must be refixed. However, the only refixed points of $N(U)$
are the endpoints of $M(U)$. Thus, $\Pi(c)=U'$.

(2) This case is left to the reader.
\end{proof}

The opposite claim is also true.

\begin{lem}\label{l:perpic}
If $c=\ol{xy}$ is a critical leaf with a periodic endpoint, say, $y$
such that $x, y\in \Pi(c)$ then there exists a quadratic invariant
gap $U$ with a periodic major $M(U)=\ol{yz}$ of period $k$ such that
the properties listed in Lemma~\ref{picper} hold and $z$ is the
closest to $y$ in $L(c)=(y, x)$ point which is $\si^k$-fixed
(moreover, $z$ is of minimal period $k$).
\end{lem}

\begin{proof}
Assume that the period of $y$ is $k$. Consider the gap $G(c)$. By
the properties of gaps it follows that there is an edge $c_{-1}$ of
$G(c)$ attached to $c$ which maps to $c$ under $\si^k$. Moreover,
$|c_{-1}|_{G(c)}=3^{-k-1}$. This can be continued infinitely many
times so that the $m$-th edge of $G(c)$ which maps to $c$ under
$\si^{mk}$ is a leaf $c_{-m}$ such that its hole is of length
$3^{-mk-1}$. Clearly, the concatenation $A$ of leaves $c$, $c_{-1}$,
$\dots$ converges to a point $z\in \uc$ which is $\si^k$-fixed.
Set $\ol{zy}=M$.

Since by Lemma~\ref{pic} $G(c)$ is a quadratic gap, then $\ol{A}$
has a lot of preimages on the boundary of $G(c)$.
Let us replace all of them by the corresponding preimages of $M$
(in particular, let us replace $\ol{A}$ itself by $M$).
It follows that the newly
constructed gap $U$ is  a quadratic gap with  major $M=M(U)$.
Clearly, this is precisely the situation described in
Lemma~\ref{picper} as desired. The last claim of the lemma easily
follows.
\end{proof}

We can now summarize the above proven results.
Let $c$ be a critical leaf.
Then the orbit of $c$ can be of three distinct types.
Firstly, the orbit of $c$ can be contained in $L(c)$.
Then the gap $G(c)=U(c)$ and the leaf $c$ are called \emph{regular critical}.
Secondly, an endpoint of $c$ can be periodic with the orbit of $c$
contained in $\ol{L(c)}$.
As described in Lemmas~\ref{picper} and ~\ref{l:perpic}, then $\Pi(c)$ consists of
$U'$, $C'(c)$, and all (iterated) pullbacks of $C'(c)$ to holes of
$U$ for some invariant quadratic gap $U=U(c)$ with a periodic major
$M(U)$ (the gap $U$ can be defined as the convex hull of all
non-isolated points in $\Pi(c)$). Then we call the gap $G(c)$ an
\emph{extended \ca\, gap}, and the critical leaf $c$ a \emph{\ca\,}
critical leaf.
Thirdly, there can be $n>0$ with $\si^n(c)\nin \ol{L(c)}$.
Denote by $n_c$ the smallest such $n$.
Then we call $G(c)=U(c)$ a \emph{gap of periodic type} and we call $c$ a
critical leaf of \emph{periodic type}.

Given a periodic leaf $\ell$,
when can we guarantee by just looking at is dynamics that $\ell$ is
in fact the major of an invariant quadratic gap of periodic type?
We show below in Lemma~\ref{l:majordesc} that properties listed in
Lemma~\ref{bndcrit} are basically sufficient for $\ell$ to be
the major of this kind.
First we need a helpful observation.

\begin{lem}\label{l:posholes}
Suppose that $\ell=\ol{xy}$ is a $\si_d$-periodic leaf of period $r$
such that its endpoints are $\si_d^r$-fixed. Moreover, suppose that
there is a unique component $Z$ of $\ol{\disk}\sm
\orb_{\si_d}(\ell)$ such that all images of $\ell$ are pairwise
disjoint boundary leaves of $Z$ and $(x, y)$ is the hole of $Z$ behind
$\ell$. Then the following holds.

\begin{enumerate}

\item If $d=2$ then $Z$ is a semi-laminational set;

\item If $d=3$ and there is a hole of $Z$ with length
greater than $\frac13$, then
$Z$ is a semi-laminational set.
\end{enumerate}

\end{lem}

\begin{proof}
If $I=(\si_d^k(x), \si_d^k(y))$ is a hole of $Z$ while
$(\si_d^{k+1}(y), \si_d^{k+1}(x))$ is a hole of $Z$, then we say
that $\si_d$ \emph{changes orientation} on $I$. Now let us show that
if $I=(\si_d^k(x), \si_d^k(y))$ is a hole of $Z$ such that $\si_d$
changes orientation on $I$, then $I$ contains a $\si_d$-fixed point.
Indeed, the image of $I$ is an arc which connects $\si_d^{k+1}(x)$
to $\si^{k+1}(y)$ and potentially wraps around the circle a few
(zero or more) number of times. If $(\si_d^{k+1}(y),
\si_d^{k+1}(x))$ is a hole of $Z$, it follows that $I\subset
\si_d(I)$ and that the endpoints of $\si_d(I)$ do not belong to $I$.
Hence $I$ contains a $\si_d$-fixed point. Now let us consider two
separate cases.

(1) Suppose that $d=2$.
Then by the above $\si_2$ can change
orientation on no more than one hole. However since the leaf is
periodic and its endpoints are refixed, $\si_2$ has to change
orientation an even number of times, a contradiction.

(2) Suppose that $d=3$ and $(x, y)=H$ is a hole of $Z$ of length
greater than $\frac13$.
Choose a critical leaf $c$ whose endpoints
are in $(x, y)$ and are non-periodic. Consider the gap $G(c)$. Then
either $\ell$ is the major of $G(c)$, in which case $G(c)$ is of
periodic type and the claim follows from the properties of such gaps.
Otherwise, it is not hard to verify that the $\psi_{G(c)}$-images of
$\ell$ and its images are leaves forming a $\si_2$-periodic orbit
of leaves with the properties from (1).
\end{proof}

\begin{lem}\label{l:majordesc}
Suppose that $\ell$ is a periodic leaf, and there is a unique
component $Z$ of $\ol{\disk}\sm \orb(\ell)$ such that all images of
$\ell$ are pairwise disjoint boundary leaves of $Z$. Moreover, suppose that the hole
$H_Z(\ell)$ of $Z$ behind $\ell$ is such that
$\frac13<|H_U(\ell)|<\frac23$ while the other holes of $Z$ behind
images of $\ell$ are shorter than $\frac13$. Then $\ell$ is the
major of an invariant quadratic gap $U=\ch(\Pi(\ell))$.
\end{lem}

\begin{proof}
Let us draw a critical leaf $c$ through an endpoint $x$ of
$\ell=\ol{xy}$ so that the $x$ and $x+\frac13$ are the endpoints of
$c$. By the Lemma~\ref{l:perpic} this gives rise to an extended
caterpillar gap as well as a gap of periodic type $U$ with periodic
major $M(U)$. We need to show that $M(U)=\ell$. To this end consider
the leaf $\ell''=\ol{(y-\frac13)(x+\frac13)}$. Observe by
Lemma~\ref{l:posholes} the map does not change orientation on any
hole of $Z$. Then the arcs $[x, y-\frac13]$ and $[x+\frac13, y]$
have the same image-arc $I=[\si(x), \si(y)]$; moreover, by the
assumptions of the lemma after that the arc $I$ maps forward in the
one-to-one fashion until it maps onto $[x, y]$. This shows that $y$
is the closest to $x$ (in the positive direction) $\si^k$-fixed point
beyond $x+\frac13$ and implies, by Lemma~\ref{vassal}, that $\ell$
is the (periodic) major of $U$ defined as in Lemma~\ref{l:perpic}.
\end{proof}

By the above proven lemmas, each gap $U=G(c)$ of periodic type has a periodic major
$M(U)=\ol{xy}$ of period $n_c$ with endpoints in $L(c)$; moreover,
$x$ and $y$ are the closest in $L(c)$ points to the endpoints of
$c$, which are $\si^{n_c}$-fixed (in fact, they are periodic of
minimal period $n_c$).

\begin{lem}\label{cantor}
If  an invariant  quadratic gap $U$ is either regular critical or of
periodic type, then $U'$ is a Cantor set. If $U=G(c)$ is an extended
caterpillar gap, then $U'$ is the union of a Cantor set and a
countable set of isolated points all of which are preimages of the
endpoints of $c$.
\end{lem}

\begin{proof}
In the regular critical and periodic cases, it suffices to prove
that the set $U'$ has no isolated points. Indeed, iterated preimages
of the endpoints of $M(U)$ are dense in $U'$ by the expansion properties of $\si$.
Therefore, an isolated point in $U'$ must
eventually map to an endpoint of $M(U)$. Thus it remains to show
that the endpoints of $M(U)$ are not isolated. This follows because
the endpoints of $M(U)$ are periodic, and suitably
chosen pullbacks of points in $U'$ to $U'$ under the iterates of the
remap of $U'$ will converge to the endpoints of $M(U)$.

The case of an extended caterpillar gap follows from Lemma~\ref{picper}.
\end{proof}

\subsection{Canonical laminations of invariant quadratic gaps}
\label{s:canlamqua}

We now associate a specific lamination with every invariant
quadratic gap. This is done in the spirit of \cite{thu85} where
pullback laminations are defined for a given, maximal collection of
critical leaves. Since classes of a lamination $\sim$ are finite, an
invariant lamination cannot contain
 an (extended) \ca{} gap. Hence below we consider
only regular critical gaps and gaps of periodic type.

Suppose that $U$ is a stand alone quadratic invariant gap of regular
critical type, i.e. the major $C=M(U)$ is a critical leaf. Edges of
$U$ have uniquely defined iterated pullbacks that are disjoint from
$U$. These pullbacks define an invariant lamination. More precisely,
we can define a lamination $\sim_U$ as follows: two points $a$ and
$b$ on the unit circle are equivalent if either they are endpoints
of an edge of $U$, or there exists $N>0$ such that $\si^N(a)$ and
$\si^N(b)$ are endpoints of the same edge of $U$, and the set
$\{\si^i(a),\si^i(b)\}$ is not separated by $U$ for $i=0$, $\dots$,
$N-1$.
Note that all classes of $\sim$ are either points or leaves.
Clearly $\sim_U$ is an invariant lamination.

Let $U$ be of periodic type and $V$ be its vassal. Define a lamination
$\sim_U$ as follows: two points $a, b\in \uc$ are equivalent if
either $a$ and $b$ are endpoints of an edge of $U$ or $V$, or  there
exists $N>0$ such that $\si^N(a)$ and $\si^N(b)$ are endpoints of
the same edge of $U$ or the same edge of $V$, and the chord
$\ol{\si^i(a)\si^i(b)}$ is disjoint from $U\cup V$ for $i=0$,
$\dots$, $N-1$. Note that $V$ is a gap of $\sim_U$.

\begin{lem}\label{l:canlam1}
If $U$ is a stand alone invariant quadratic gap of regular critical type, then
$\sim_U$ is the unique invariant lamination such that $U$ is
one of its gaps.
\end{lem}

A proof of this lemma will be given after we establish one general
result.
We say that a
lamination $\sim$ {\em co-exists with} a laminational set $G$ if no
leaves of $\sim$ intersect the edges of $G$ in $\disk$.
We say that a lamination $\sim$ co-exists with a lamination
$\approx$ if no leaf of $\sim$ intersects a leaf of $\approx$ in $\disk$.

\begin{lem}
  \label{l:canlam1s}
  Suppose that a cubic invariant lamination $\sim$ co-exists with
  a stand alone invariant quadratic gap $U$ of regular critical type.
  Then $\sim$ also co-exists with the canonical lamination $\sim_U$ of $U$.
\end{lem}

\begin{proof}
  Suppose that a leaf $\ell$ of $\sim$ crosses a leaf $\ell_U$ of $\sim_U$ in $\disk$.
  By the assumption of the lemma, this cannot happen if $\ell$ or $\ell_U$
  intersects $U$.
  Hence both $\ell$ and $\ell_U$ have their endpoints in some hole of $U$.
  Every hole of $U$ maps one-to-one onto its image.
  It follows that $\si(\ell)$ and $\si(\ell_U)$ also intersect in $\disk$.
  However, any leaf of $\sim_U$ eventually maps to an edge of $U$,
  a contradiction.
\end{proof}

\begin{proof}[Proof of Lemma \ref{l:canlam1}]
Suppose that $\sim$ is a cubic invariant lamination such that $U$ is
a gap of $\sim$. It follows that $\sim$ co-exists with $U$, hence by
Lemma \ref{l:canlam1s} it also co-exists with the canonical
lamination $\sim_U$. If there is a leaf $\ell$ of $\sim$ not
contained in $\sim_U$, then this leaf is in some pullback of $U$.
Hence the leaf $\ell$ eventually maps to $U$. Since $U$ is a
gap of $\sim$, the leaf $\ell$ eventually maps to an edge of
$U$. It follows that $\sim$ is a leaf of $\sim_U$, since pullbacks
of edges of $U$ are well defined. For the same reason, all leaves of
$\sim_U$ are also leaves of $\sim$.
\end{proof}

The proof of Lemma~\ref{l:canlam2} is similar to that of Lemmas
~\ref{l:canlam1} and \ref{l:canlam1s}.

\begin{lem}
  \label{l:canlam2}
Let $U$ be an invariant quadratic gap of periodic type. Then
$\sim_U$ is the unique invariant lamination such that $U$ and the
vassal $V(U)$ are among its gaps. Moreover, if a cubic invariant
lamination $\sim$ co-exists with $U$ and $V(U)$ then $\sim$ also
co-exists with the canonical lamination $\sim_U$ of $U$.
\end{lem}

\begin{dfn}\label{d:fa}
Consider the set of all points with orbits \emph{above}
$\di$ (it is easy to see that this is a Cantor set). The convex hull
of this set is a forward invariant gap which maps onto itself two-to-one.
It is denoted $\fg_a$ (``a'' from ``above'').
Similarly, define an invariant gap $\fg_b$ whose basis consists of
all points with orbits \emph{below} $\di$ (``b'' in $\fg_b$
comes from ``below''). Both gaps $\fg_a, \fg_b$ together with their
pullbacks into holes of their bases define a lamination
$\sim_{\di}$ (or $\lam(\di)$) called the \emph{canonical lamination}
of $\di$ (the construction is classical for complex combinatorial
dynamics, see \cite{thu85}).
\end{dfn}

\subsection{Cubic laminations with no rotational gaps or leaves}

We now describe all cubic laminations with no periodic rotational
gaps or leaves (a {\em rotational leaf} is a periodic leaf with
non-refixed endpoints).

\begin{lem}
\label{l:pc0}
 Suppose that a cubic invariant lamination $\sim$ has no periodic rotational
 gaps or leaves.
Then either $\sim$ is empty, or it coincides with the canonical lamination of
an invariant quadratic gap.
\end{lem}

In fact, we need a more general result. Recall that for an invariant
quadratic gap of periodic type $U$ we denote by $M''(U)$ the edge of
$V(U)$ with the same image as the major $M(U)$.

\begin{lem}
 \label{l:spec-GvsL} Let $U$ be a stand alone invariant quadratic
 gap. The following claims hold.

 \begin{enumerate}
 \item If $U$ is of periodic type and both $M(U)$ and $M''(U)$ co-exist with $\sim$
 but are not contained in the same gap of $\sim$, then $\sim$ has a
 rotational gap or leaf in $V(U)$.

 \item If $\sim$ does not co-exist with the canonical lamination
$\sim_U$ of $U$ but co-exists with $U$, then $\sim$ has a rotational
gap or leaf in $V(U)$, and $M(U)$ and $M''(U)$ are leaves of $\sim$.

 \end{enumerate}
\end{lem}

\begin{proof}

(1) Suppose that $M(U)=M$ and $M''(U)=M''$ are not contained in the same
gap. Let $T\supset M$ and $T''\supset M''$ be the gaps or leaves of
$\sim$ with the same image; by the assumption, $T\ne T''$. By
Lemma~\ref{l:easytop} there is a critical gap or leaf $C$ contained
in the strip $S(M, M'')$ between $M$ and $M''$ and distinct from
either $T$ or $T''$; by the assumption, $C\ne V(U)$ and hence
$\si^k(C)\ne C$. If $\si^k(C)$ is non-disjoint from $C$, then $C$ is
a periodic infinite gap. Consider the concatenation $X$ of $C$ with
its $\si^k$-images. Because of the dynamics of $M$ it follows that
$X$ cannot contain $T$. Hence $X$ is contained in $S(M, M'')$ which
implies that $C\subset V(U)$. If $C$ is Siegel gap, we are done.
Otherwise $C$ is a periodic Fatou gap of degree greater than 1 which
easily implies that, since $C$ and $\si^k(C)$ are not disjoint, there is
a finite rotational gap or leaf which borders $X$, and we are done.

Now let $\si^k(C)$ be disjoint from $C$. Observe that $\si^k$ repels
$C$ away from $M$. Thus, there is a hole $H$ of $C$, disjoint from
the strip between $M$ and $C$, and such that $\si^k(C)\subset
\ol{H}$. If $H$ does not contain $M''$, then $H$ is located between
$M$ and $M''$. By Theorem~\ref{t:fxpt} applied to $H$ itself there
is a $\si^k$-fixed gap or leaf $Z$ with the basis in $H$. If $H$
contains $M''$, we apply Theorem~\ref{t:fxpt} to the strip between
$C$ and $M''$ and find the $\si^k$-fixed gap or leaf $Z$ again.
Thus, $Z\subset S(M, M'')$. Since $\si^k(Z)=Z$, we see that
$Z\subset V(U)$. If $Z$ is rotational, we are done. Otherwise $Z$ is
a Fatou gap of degree greater than $1$. Hence $C$ is a gap from
the orbit of $Z$ and $\si^k(C)=C$, a contradiction.

(2) Easily follows from (1).
\end{proof}

\begin{proof}[Proof of Lemma \ref{l:pc0}] By
Theorem \ref{t:fxpt} there exists an invariant $\lam_\sim$-set $U$.
Since $U$ is not a rotational gap or leaf, it must be either a fixed
non-rotational leaf or an invariant Fatou gap of degree $>1$
(it cannot be a fixed finite non-rotational gap because there
are only two $\si$-fixed points in $\uc$, namely $0$ and $1/2$).
There is only one fixed non-rotational leaf, namely, $\ol{0\frac 12}$.

Suppose first that $\lam_\sim$ contains the leaf $\ol{0\frac 12}$.
By the invariance property of $\sim$, both leaves
$\ell_a=\ol{\frac16\frac13}$ and $\ell_b=\ol{\frac23\frac56}$ are leaves of $\sim$
(there is no other way to connect the four points $\frac 16$, $\frac 13$,
$\frac 23$ and $\frac 56$ by leaves).

We claim that if $\di$ is isolated from above then $\fg_a$ is a gap
of $\sim$. Indeed, if $\di$ is isolated from above, then there is a
$\sim$-gap $G_a$ located above $\di$. It cannot be finite, hence it
is a quadratic invariant gap. The description of all such gaps given
in Subsection~\ref{s:invquagap} implies that $G_a=\fg_a$. Similarly,
if $\di$ is isolated from below, $\fg_b$ is a $\sim$-gap. Hence if
$\di$ is isolated from both sides, both $\fg_a$ and $\fg_b$ are
$\sim$-gaps. By Lemma~\ref{l:canlam2}, then $\sim$ coincides with the
canonical lamination $\sim_{\fg_a}=\sim_{\fg_b}$. Suppose that $\di$
is not isolated from above. Then by Lemma~\ref{l:spec-GvsL}(2) we
get a contradiction with our assumption that $\lam_\sim$ has no
rotational gaps or leaves.

Now assume that $U$ is an invariant Fatou gap of degree $>1$ with no
fixed edges. If the degree of $U$ is three, $\sim$ is empty. Let the
degree of $U$ be two. By Lemma \ref{l:spec-GvsL}(3), the fact that
$\sim$ does not have any rotational gaps or leaves implies that
$\sim$ co-exists with the canonical lamination $\sim_U$. If $U$ is
of regular critical type, then the statement now follows from Lemma
\ref{l:canlam1}. If there are no leaves of $\sim$ in the vassal
$V(U)$, then the statement follows from Lemma \ref{l:canlam2} and
the fact that $U$ is a gap of $\sim$. If there are leaves of $\sim$
in $V(U)$, then by Lemma \ref{l:spec-GvsL} there is a rotational gap
or leaf in $V(U)$, a contradiction.
\end{proof}

\section{Invariant rotational sets}\label{s:rotgaps}

In this section, we study invariant rotational laminational sets.
Fix any such set (a stand alone gap) $G$.
There are one or two majors of $G$.
We classify invariant rotational gaps by types.
This classification roughly follows Milnor's classification of
hyperbolic components in slices of cubic polynomials and quadratic
rational functions \cite{M, miln93}.
Namely, we say that
\begin{itemize}
  \item the gap $G$ is of type A (from ``Adjacent'') if $G$ has only one
  major (whose length is at least $\frac 23$);
  \item the gap $G$ is of type B (from ``Bi-transitive'') if $G$
  has two majors that belong to the same periodic cycle;
  \item the gap $G$ is of type C (from ``Capture'') if it is not type B,
  and one major of $G$  eventually maps to  the other major of $G$;
  \item the gap $G$ is of type D (from ``Disjoint'') if there are two
  majors of $G$, whose orbits are disjoint.
\end{itemize}
Clearly, finite rotational gaps cannot be of  type C, and infinite
rotational gaps cannot be of  type B.

\subsection{Finite rotational sets}\label{s:finrotset}
Let $G$ be a finite invariant rotational set. By \cite{kiwi02},
there are either one or two periodic orbits forming the set of
vertices of $G$. These periodic orbits are of the same period
\emph{denoted in this section by $k$} (as we fix $G$ in this
section, we omit using $G$ in the notation). If vertices of $G$ form
two periodic orbits, points of these orbits alternate on $\uc$,
thus we can talk of one or two periodic orbits of edges of $G$
(in the latter case, $G$ is of type D). We now give examples.

\begin{example}\label{fingap1}
Consider the invariant rotational gap $G$ with vertices
$\frac7{26}$, $\frac4{13}$, $\frac{11}{26}$, $\frac{10}{13}$,
$\frac{21}{26}$ and $\frac{12}{13}$.
This is a gap of type D.
The major leaf $M_1$ connects $\frac{12}{13}$ with $\frac{7}{26}$ and
the major leaf $M_2$ connects $\frac{11}{26}$ with $\frac{10}{13}$.
These major leaves belong to two distinct periodic orbits of edges of $G$.
The major hole $H_G(M_1)$ contains $0$ and the major hole
$H_G(M_2)$ contains $\frac12$.
\end{example}

\begin{figure}
\includegraphics[width=6cm]{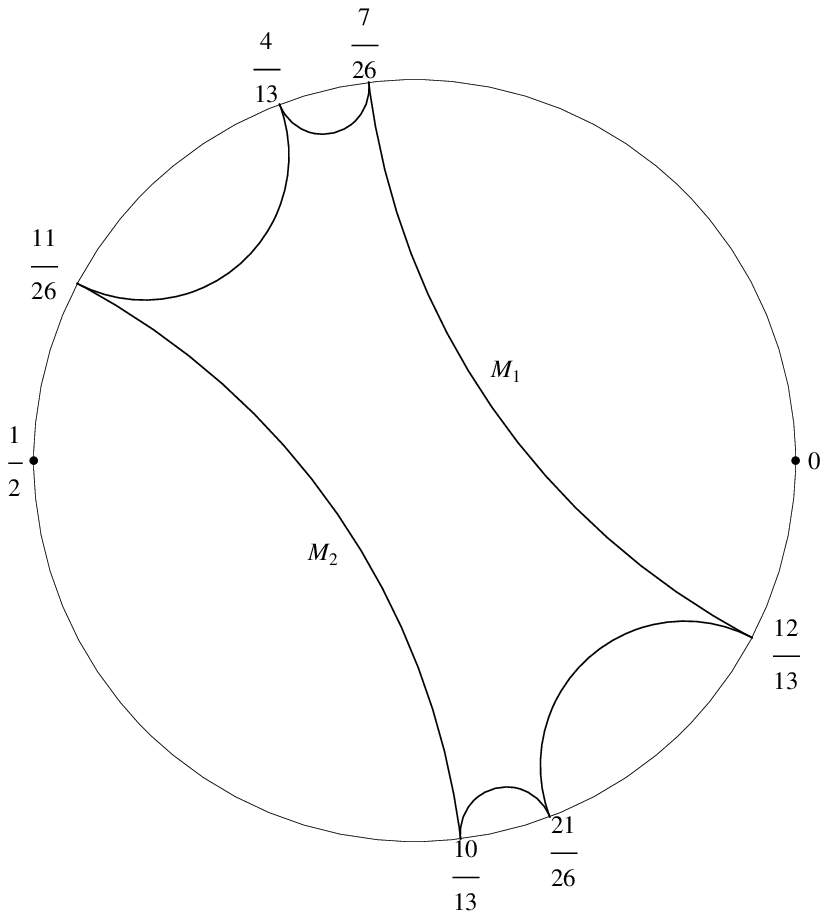}
\includegraphics[width=6cm]{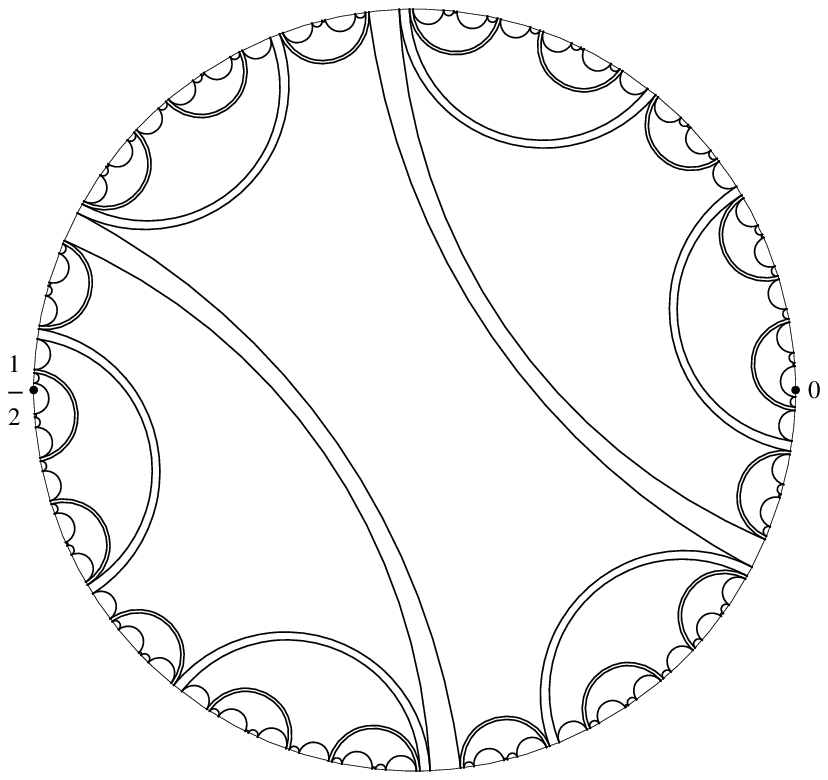}
\caption{The rotational gap described in Example \ref{fingap1} and
its canonical lamination.}
\end{figure}

The next example can be obtained by considering \emph{one} periodic orbit
of vertices in the boundary of the gap from Example~\ref{fingap1}

\begin{example}\label{fingap2}
Consider the finite gap $G$ with vertices $\frac7{26}$,
$\frac{11}{26}$ and $\frac{21}{26}$.
This is a gap of type B.
The first major leaf $M_1$ connects $\frac{21}{26}$ with $\frac{7}{26}$
and the second major leaf $M_2$ connects $\frac{11}{26}$ with $\frac{21}{26}$.
The edges of $G$ form one periodic orbit to which
both $M_1$ and $M_2$ belong.
The major hole $H_G(M_1)$ contains $0$ and the major hole
$H_G(M_1)$ contains $\frac12$.
\end{example}


\begin{figure}
\includegraphics[width=6cm]{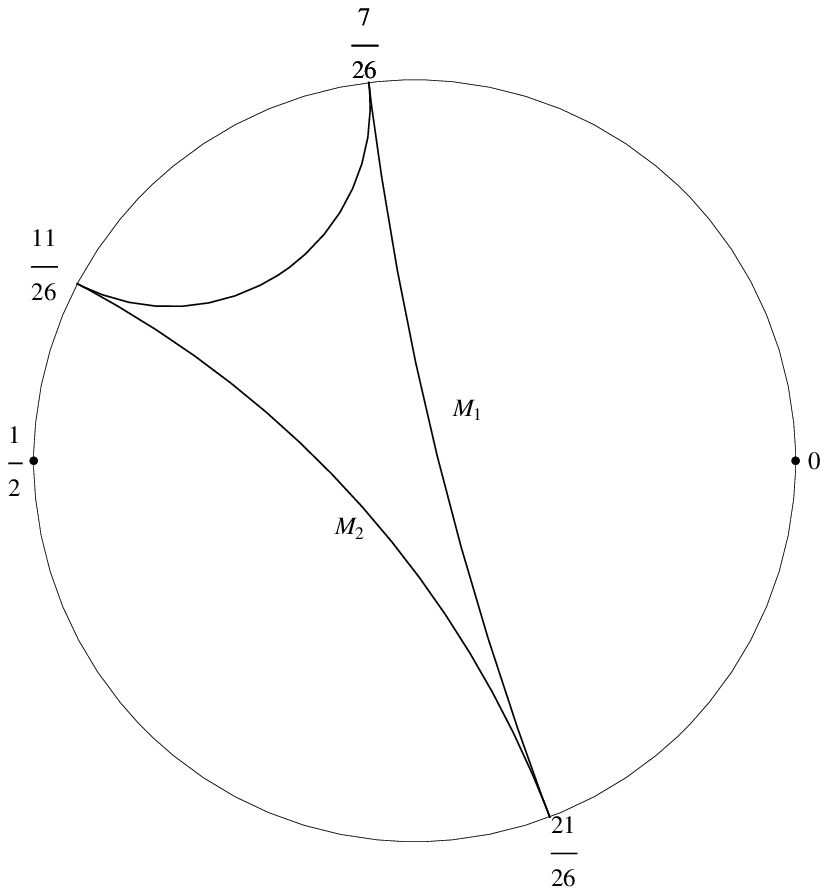}
\includegraphics[width=6cm]{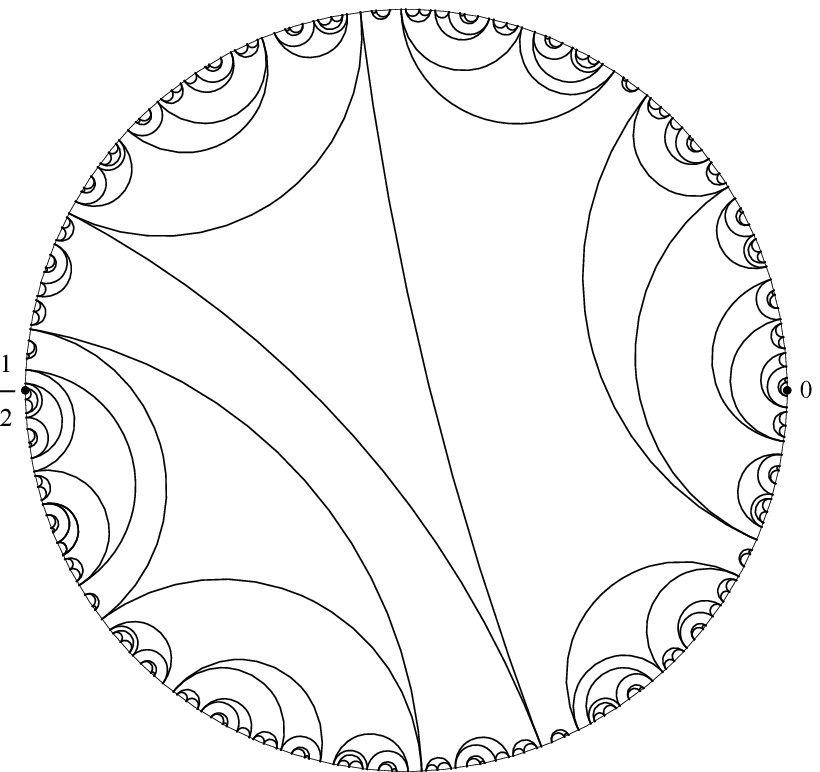}
\caption{The rotational gap described in Example \ref{fingap2} and
its canonical lamination.}
\end{figure}


\begin{example}\label{fingap3}
Consider the finite gap $G$ with vertices $\frac1{26}$,
$\frac{3}{26}$ and $\frac{9}{26}$.
This is a gap of type A.
The only major leaf $M=M_1=M_2$ connects $\frac{9}{26}$ with $\frac{3}{26}$.
The edges of $G$ form one periodic orbit to which $M$ belongs.
The major hole $H_G(M)$ contains $0$ and $\frac12$ and is longer than $\frac23$.
\end{example}


\begin{figure}
\includegraphics[width=6cm]{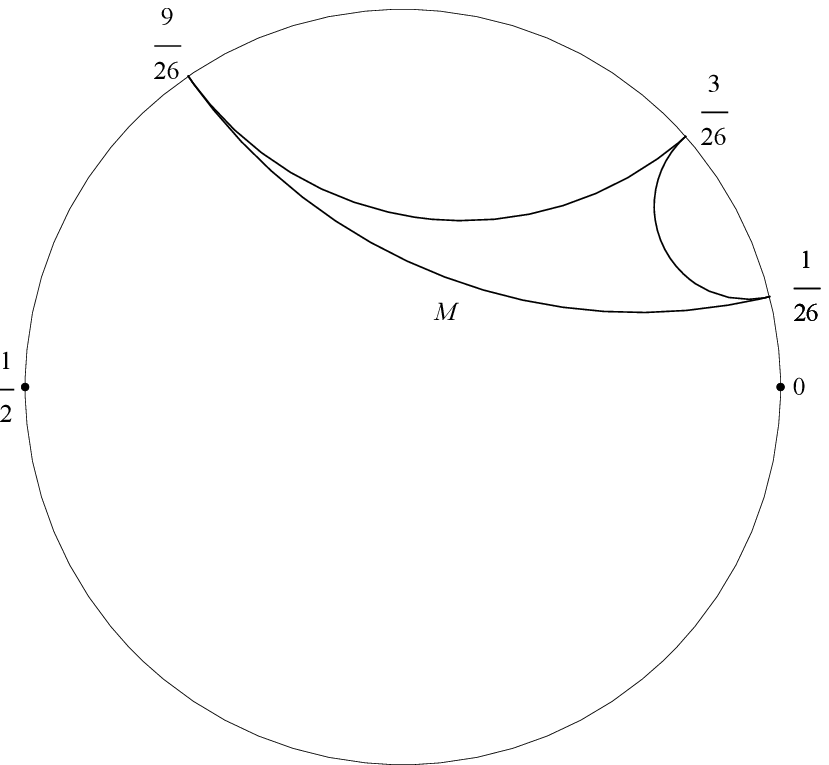}
\includegraphics[width=6cm]{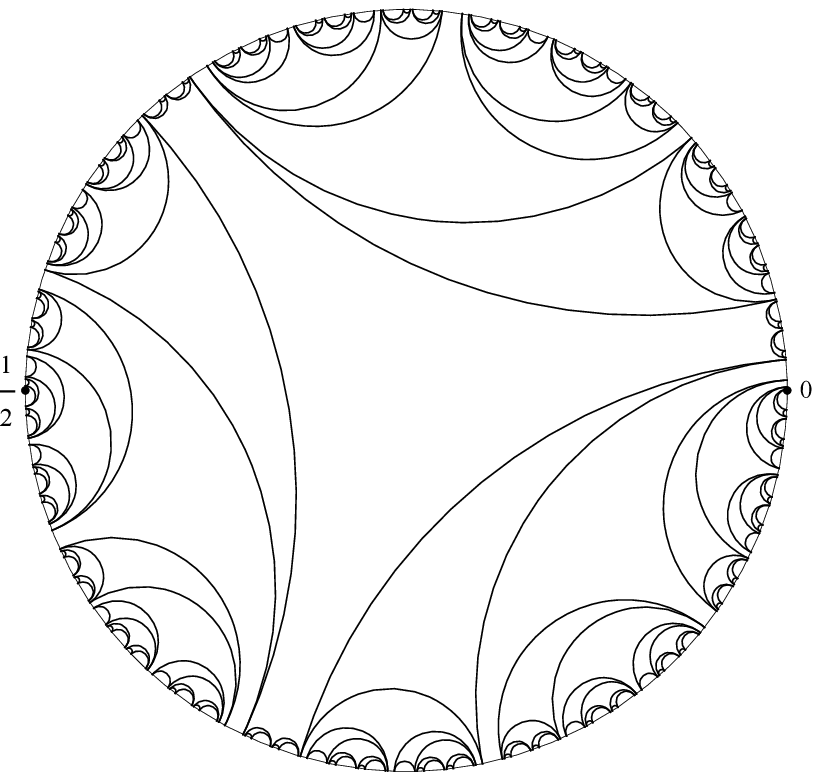}
\caption{The rotational gap described in Example \ref{fingap3} and
its canonical lamination.}
\end{figure}

Let $G=\ell=\ol{ab}$ be an invariant leaf. We can think of $G$ as a
gap with empty interior and two edges $\ol{ab}$ and $\ol{ba}$, and
deal with all finite periodic rotational sets in a unified way. Let
us list all invariant leaves $\ol{ab}$. Either points
$a$, $b$ are fixed, or they form a two-periodic orbit. In the first
case, we have the leaf $\ol{0\frac12}=\di$, in the second case, we
have one of the leaves $\ol{\frac18\frac38}, \ol{\frac14\frac34},
\ol{\frac58\frac78}$. Informally, we regard $\di$ as an invariant
rotational set of type D (even though its rotation number is $0$).

Let $G$ be a finite invariant laminational set with $m$ edges
$\ell_0$, $\dots$, $\ell_{m-1}$. For each $i$, let $\fg_i$ be the
convex hull of all points $x\in \ol{H_G(\ell_i)}$ with
$\si^j(x)\in \ol{H_G(\si^j(\ell_i))}$ for every $j\ge 0$
(compare this to the definition of a vassal in Section
\ref{s:qgaps}). It is straightforward to see that $\fg_i$ are
infinite stand alone gaps such that $\fg_i$ maps to $\fg_j$ if
$\ell_j=\si(\ell_i)$. These gaps are called the {\em canonical
Fatou gaps attached to $G$}. The gap $\fg_i$ is critical if and only
if the corresponding edge $\ell_i$ is a major.

\subsection{Canonical laminations of finite invariant rotational sets}
\label{s:canlamfin}

To every finite invariant rotational set $G$, we associate its
\emph{canonical} lamination $\sim_G$.

Suppose first that $G$ is of type B or D.
Then by definition there are two major edges of $G$, which we denote by
$M_1$ and $M_2$.
Let $H_1$ and $H_2$ be the corresponding holes.
By Lemma \ref{l:fx-maj}, one of the holes contains $0$, and the other
contains $\frac 12$ (say, $0\in H_1$ and $\frac 12\in H_2$). Since
$M_1$ and $M_2$ are periodic, their lengths are strictly greater than $\frac 13$.
Let $U_1$ and $U_2$ be the canonical Fatou gaps attached to $M_1$
and $M_2$, respectively.

A standard pullback construction of Thurston \cite{thu85} allows us
to construct an invariant lamination by pulling back the gap $G$ so
that the pullbacks are disjoint from the interiors of $U_1$ and
$U_2$. More precisely, we can define a lamination $\sim_G$ as
follows: two points $a$ and $b$ on the unit circle are equivalent if
there exists $N\ge 0$ such that $\si^N(a)$ and $\si^N(b)$ are
vertices of $G$, and the chords $\ol{\si^i(a)\si^i(b)}$ are disjoint
from $G$ and from the interior of $U_1\cup U_2$ for $i=0$, $\dots$,
$N-1$. It is straightforward to check that $\sim_G$ is indeed an
invariant lamination. This lamination is called the \emph{canonical
lamination associated with $G$}.

Assume now that $G$ is of type A.
Let $M$ be the major edge of $G$, and $U$ the corresponding
canonical Fatou gap attached to $G$ at $M$.
The canonical lamination $\sim_G$ of $G$ is defined similarly to types B and D.
Namely, two points $a$ and $b$ on the unit circle are equivalent
with respect to $\sim_G$ if there exists $N\ge 0$ such that
$\si^N(a)$ and $\si^N(b)$ are vertices of $G$, and the chord
$\ol{\si^i(a)\si^i(b)}$ is disjoint from $G$ and from the interior of $U$
for $i=0$, $\dots$, $N-1$.

The following lemma is proven similarly to Lemma~\ref{l:canlam1} and
is in fact based upon Thurston's pull back construction of
laminations which have a full collection of critical gaps and
leaves.

\begin{lem}\label{l:gapsuniq}
Suppose that $\sim$ has a finite invariant gap $G$ and all the
canonical Fatou gaps attached to $G$. Then $\sim$ coincides with the
canonical lamination of $G$.
\end{lem}

\begin{lem}
 \label{l:spec-LvsG1} Suppose that a cubic invariant lamination
$\sim$ has a finite invariant gap $G$ of type D. Then,
if a canonical Fatou gap $U$ of $G$ is not a gap of $\sim$, then
$\sim$ has a rotational gap or leaf in $U$. In particular, if $\sim$
does not coincide with the canonical lamination of $G$, then $\sim$
has a rotational periodic gap or leaf in a canonical Fatou gap
attached to $G$.
\end{lem}

\begin{proof}
Observe, that a major of a gap $G$ of type D satisfies the
conditions of Lemma~\ref{l:majordesc}. Therefore it can be viewed as
the major of some invariant quadratic gap $W$ of periodic type. It
follows that the canonical Fatou gap $U$ attached to $G$ coincides
with the vassal gap $V(W)$ of $W$. Hence the rest of the lemma
follows from Lemmas~\ref{l:spec-GvsL} and \ref{l:gapsuniq}.
\end{proof}

\subsection{Irrational invariant gaps}\label{s:irrgaps}

The description of irrational gaps is reminiscent of that of finite
laminational sets. In this subsection, we fix an irrational rotation
number $\tau$.

Let $G$ be a rotational invariant gap of rotation number $\tau$
(i.e., $G$ is a Siegel gap). Then $G$ may have one or two critical
majors of length $\frac13$, or one critical major of length
$\frac23$. It is also possible that $G$ has a non-critical major.
However, a non-critical major eventually maps to a critical major by
Lemma \ref{l:maj} (in this case, $G$ is of type C).
Thus an infinite rotational gap $G$ can have type A, C or D.

\subsection{Canonical laminations of irrational invariant gaps}\label{s:canlamirr}

We now construct the canonical lamination for a Siegel gap $G$.
Let $G$ be a gap of type D with critical edges $L$ and $M$.
Consider well-defined pullbacks of $G$ attached to $G$ at $L$ and $M$.
Then apply a standard pullback procedure \cite{thu85} to these gaps.
As holes in the union of bases of these gaps are shorter than
$\frac13$, the pullbacks of the gaps converge in diameter to $0$.
Alternatively, we can define $\sim_G$ as follows: two points $a$ and
$b$ on the unit circle are equivalent if there exists $N>0$ such
that $\si^N(a)$ and $\si^N(b)$ lie on the same edge of $G$, and the
chords $\ol{\si^i(a)\si^i(b)}$ are disjoint from $G$ for $i=0$, $\dots$, $N-1$.

Suppose now that $G$ is of type C. Let $M$ be the critical major of
$G$, and $Q$ be the critical quadrilateral (i.e. quadrilateral with
critical diagonals), one of whose edges, denoted by $R$, is the non-critical major of
$G$. We now define $\sim_G$ as follows: two points $a$ and $b$ on
the unit circle are equivalent if either $a$ and $b$ are endpoints
of an edge of $G$, 
or there exists $N>0$ such that $\si^N(a)$
and $\si^N(b)$ are endpoints of the same edge of $G$, 
and the chords $\ol{\si^i(a)\si^i(b)}$ are disjoint from $G\cup Q$
for $i=0$, $\dots$, $N-1$. It is easy to see that $\sim_G$ is a
lamination. Moreover, there exists a $\sim_G$-gap $T_G$ attached to
$G$ at $R$ which maps onto its image in a two-to-one fashion. The
rest of the lamination $\sim_G$ consists of concatenated to $G$ and
$T_G$ pullbacks of these gaps. Observe that $M$ and $T_G$ are the only
critical gaps (leaves) of $\sim_G$.

Consider now a Siegel gap $G$ of type A. Let $M$ denote its major.
By Lemma~\ref{l:fx-maj}, we have $0,\frac12\in H_G(M)$, and
therefore $G$ is located either completely above $\di$ (then
$G\subset\fg_a$), or completely below $\di$ (then $G\subset \fg_b$).
Thurston's pullback construction applies to this case as well.
Define $Q$ as the critical equilateral triangle, one of whose edges
is $M$. We now define $\sim_G$ as follows: two points $a$ and $b$ on
the unit circle are equivalent if either $a$ and $b$ are endpoints
of an edge of $G$, or there exists $N>0$ such that $\si^N(a)$ and
$\si^N(b)$ lie on the same edge of $G$, 
and the chords
$\ol{\si^i(a)\si^i(b)}$ are disjoint from $G\cup Q$ for $i=0$,
$\dots$, $N-1$.
It is easy to see that $\sim_G$ is a
lamination. Moreover, there exists a $\sim_G$-gap $T_G$ attached to
$G$ at $M$ which maps onto its image in a two-to-one fashion. The
rest of the lamination $\sim_G$ consists of concatenated to $G$ and
$T$ pullbacks of these gaps. Observe that $M$ and $T_G$ are the only
critical gaps (leaves) of $\sim_G$.

\begin{lem}
 \label{l:type1Sie} Suppose that $G$ is a stand alone type Siegel
gap and $\sim$ is an invariant cubic lamination with gap $G$. Suppose
that one of the following holds:

\begin{enumerate}

\item $G$ is of type D;

\item $G$ is of type $A$ or $C$ and $T_G$ is a gap of $\sim_G$.

\end{enumerate}

Then $\sim$ coincides with $\sim_G$.
\end{lem}

The proof is almost verbatim like that of Lemma \ref{l:canlam1} and
is essentially based on Thurston's pullback construction of
laminations in the case when given leaves and gaps exhaust all
possible criticality.

\section{Cubic laminations with at most one periodic rotational set}

One of the aims of this paper is to describe cubic laminations
$\sim$ with the simplest non-trivial structure. As such, we consider
cubic laminations with \emph{at most} one periodic rotational gap or
leaf, satisfying one extra condition.
Namely, by
Theorem~\ref{t:corexpli} if $\sim$ has finitely many periodic
rotational sets, then the unique super-gap of $\sim$ is the entire disk.
However, among laminations with this property we can still
distinguish between the ones with more complicated dynamics and the
ones with less complicated dynamics. As a measure of that, we choose
the number of times we need to remove isolated leaves from the
lamination so that we remove all rotational objects of the
lamination.
Thus, \emph{we describe the family $\smp$ of all cubic
laminations $\sim$ with exactly one periodic rotational gap or leaf such that
all edges of this periodic rotational set are isolated}.
Equivalently, $\sim\in\smp$ if and only if $\sim$ has exactly
one rotational gap or leaf, and the first cleaning $\lam^1_\sim$ of $\lam_\sim$
has no periodic rotational gaps or leaves.

Note that if $\sim$ has a unique rotational gap or leaf then this
gap or leaf is
invariant. Since the dynamical core $\ce_{rot}$ is spanned by
periodic rotational gaps and leaves, such laminations have the
simplest dynamical core $\ce_{rot}$: a single fixed point, or
the boundary of a fixed Siegel disk or the empty set.
Observe that if $\sim$ contains a Siegel gap, then all of its boundary edges are
isolated.  To see this, note that any critical edge $\ell$ of the
Siegel gap $G$ must be isolated (a gap with image $\si(G)$ must be
on the other side of $\ell$).  Since all other edges are
preimages of critical leaves, the claim follows.
We start by
describing \emph{quadratic} laminations with the property that there is at
most one rotational set.

If $G$ is a finite $\si_2$-invariant rotational gap or leaf,
then one can define the {\em canonical lamination} of $G$
as the only quadratic invariant lamination with a
cycle of Fatou gaps attached to edges of $G$
(it represents a parabolic quadratic polynomial $z^2+c$, whose
parameter $c$ is a point on the main cardioid).
Similarly, if $G$ is a stand alone invariant Siegel gap with
respect to $\si_2$, then one can define the canonical lamination
of $G$ as the unique quadratic invariant lamination which has $G$
as a gap.

\begin{prop}\label{p:coremin-qua}
A non-empty \textbf{quadratic} lamination $\sim$ with at most one
periodic rotational gap or leaf is the canonical lamination of an
invariant rotational gap or leaf.
\end{prop}

\begin{proof}
It is easy to see that by Theorem~\ref{t:fxpt}, $\sim$ has an invariant rotational gap or leaf $G_2$.
If $G_2$ is a Siegel gap, then $\sim$ coincides with the
canonical lamination of $G_2$ because all pullbacks of $G_2$ are
uniquely defined and no Siegel gap can contain a leaf in its
interior (this is similar to Lemma \ref{l:canlam1}).
Suppose that $G_2$ is a finite gap or leaf whose vertices have period $r$.
It is well-known that $G_2$ has a unique major $M_2$ which is the edge of $G_2$ separating
the rest of $G_2$ from $0$. By \cite{kiwi02}
there is only one cycle of edges of $G_2$.
Let $V_2$ be the Fatou gap of the canonical lamination of $G_2$
which has $M_2$ as one of its edges.
Denote by $M''_2$ the edge of $V_2$ which has the same $\si_2$-image as $M_2$.
By Theorem~\ref{t:fxpt} and because $\sim$ has at most one
periodic rotational gap or leaf, the strip between $M_2$ and $M''_2$ contains
a $\si^r_2$-invariant Fatou gap $U$ of $\sim$
of degree greater than 1.
In fact, $r$ is the period of $U$ as $U$ must
be mapped under every hole of $G_2$ before returning.
Hence $U\subset V_2$ and since appropriately chosen $\si^r_2$-pullbacks of any
point of $U'$ are dense in $V'_2$, we conclude that
$U=V_2$, and $\sim$ is the canonical lamination of $G_2$.
\end{proof}

Let us go back to studying cubic laminations (recall that $\si_3$ is denoted by $\si$).
Let $\sim$ be a cubic invariant lamination.
If a lamination $\sim$ belongs to $\smp$, then it has at most one rotational set.
As explained in the beginning of this section, then
by Theorem~\ref{t:corexpli} the unique super-gap of $\sim$ is the whole $\ol{\disk}$.
Hence $\sim$ must have either a periodic Fatou
gap or a periodic Siegel gap (otherwise there are no super-gaps of $\sim$ at all).
Moreover, all edges of this periodic Fatou gap are isolated in $\lam_\sim$.

\begin{lem}\label{l:lamcoex}
Suppose that $\sim$ is a lamination which co-exists with two
disjoint critical chords, $c$ and $d$. Moreover, suppose that $c$
has non-periodic endpoints, no leaf of $\sim$ contains an
endpoint of $d$, and the leaf $d$ intersects no edge of $U(c)$ in $\disk$.
Then $\sim$ co-exists with the gap $U(c)$.
\end{lem}

\begin{proof}
We need to show that leaves of $\sim$ do not intersect the edges of
$U(c)=U$ in $\disk$.
Indeed, suppose that a leaf $\ell$ of $\sim$ intersects
an edge $\bj$ of $U$ in $\disk$.
Let us show that then $\si(\ell)$ intersects $\si(\bj)$ in $\disk$.
Indeed, there exists a component $Y_0$ of
$\ol{\disk}\sm (c\cup d)$ whose closure contains both $\ell$ and $\bj$.
The only case when $\si(\ell)$ and $\si(\bj)$ do not intersect in $\disk$
under the circumstances is as follows: $\ell$ and $\bj$
contain distinct endpoints of one of the critical chords, $c$ or $d$.
Let us show that this is impossible.

Indeed, by the assumptions on $d$ the leaf $\ell$ does not contain
an endpoint of $d$.
Suppose now that both $\ell$ and $\bj$ contain an endpoint of $c$.
Since $c\in U(c)$ and $c$ has no periodic endpoints, $c$ is regular critical.
Hence the only edge of $U(c)$ which contains an endpoint of $c$ is $c$
itself and $\ol{j}=c$.
But then $c$ and $\ol{j}$ do not intersect in $\disk$ by assumption.
Thus $\si(\ell)$ and $\si(\ol{j})$ intersect in $\disk$.
By induction and by Subsection~\ref{s:invquagap}, this implies that the
endpoints of $\ell$ belong to $U'(c)$.
Hence, $\ell$ cannot intersect $\bj$ in $\disk$, a contradiction.
\end{proof}

\begin{thm}
\label{t:cormin-spec} Let $\sim$ be a non-empty cubic invariant
lamination from $\smp$. Then one of the following cases occurs:
\begin{enumerate}
\item the lamination $\sim$ is the canonical lamination of some quadratic
invariant gap;
\item the lamination $\sim$ is the canonical lamination of some rotational
gap $G$ of type D;
\item the lamination $\sim$ co-exists with the canonical lamination of an
invariant quadratic gap $U$ and contains an invariant rotational gap
or leaf in $U$.
\end{enumerate}
\end{thm}

\begin{proof}
By Lemma~\ref{l:pc0} if $\sim$ has no rotational gap or leaf, then
(1) holds. Assume that $\sim$ has a rotational gap or leaf $G$.  By
Lemmas ~\ref{l:type1Sie} and ~\ref{l:spec-LvsG1}, if $G$ of type D,
then (2) holds.  We show that in all remaining cases we can choose
critical chords $c$ and $d$ as in Lemma~\ref{l:lamcoex}.

First consider the case when $G$ is a finite rotational gap of type A or B.
In either case, by the assumption that edges of $G$ are isolated,
there exists a cycle $\F$ of Fatou gaps attached to $G$.

Suppose that $G$ is of type B and the period of each vertex of
$G$ is $s$. Consider a few cases. It can happen that $\F$ has two
gaps, $V$ and $W$, on each of which the map $\si$ is two-to-one.
Since all gaps of $\F$ are attached to $G$, this implies that $\sim$
coincides with the canonical lamination $\sim_G$ of $G$. Indeed,
take a point $x$ in the basis of $V$ and pull it back under $\si$ following
the gaps from $\F$.
Then on each step all pullbacks of $x$ taken in the
sense $\F$ coincide with the pullbacks of $x$ taken in the sense of
$\sim_G$.
Since these pullbacks of $x$ are dense in bases of gaps in $\F$, we
obtain the desired.

In the type A case there is only one gap $V$ of $\F$ which does not
map forward one-to-one; $V$ is attached to the unique major edge of
$G$.
We may assume that $V$ is quadratic (similar to the argument above if
its remap were three-to-one, $V$ would have to coincide with the
cubic gap of $\sim_G$).

Therefore in either the type A or type B case, there is a
quadratic Fatou gap $V$ attached to some edge of $G$ and another
critical $\sim$-set $C$.  Let $c$ be a critical chord of $C$.
If by accident $c$ was chosen so that one of its endpoints
is periodic, then it follows that $C$ is a Fatou gap and $c$ can be
replaced by another critical chord inside $C$ with non-periodic
endpoints.  Note that bases of $G$ and the gaps of $\F$ are points
of $U'(c)$.  We may take $d$ to be a critical chord of $V$ whose
endpoints are not the endpoints of any leaf of $\sim$ (the basis of
$V$ is a Cantor set, so we can choose $d$ satisfying this property).
Clearly $c$ and $d$ satisfy the conditions of Lemma
~\ref{l:lamcoex}, so $\sim$ coexists with $U(c)$.
Furthermore $G\subset U(c)$.
Then by Lemma~\ref{l:spec-GvsL}, $\sim$ must coexist with $\sim_G$.
So then (3) holds.

A similar argument holds in the case that $G$ is a Siegel gap.
Let $d$ be a critical edge of $G$.
There is some other critical $\sim$-set $C$.
Let $c$ be a critical chord in $C$.  As before, $c$
may be chosen to be non-periodic and $G'\subset U'(c)$.  Since no
leaves other than $d$ of $\sim$ intersect $d$, then $c$ and $d$
satisfy the conditions of Lemma~\ref{l:lamcoex}, so that $\sim$
coexists with $U(c)$ and $G\subset U(c)$.  Then by
Lemma~\ref{l:spec-GvsL}, $\sim$ coexists with $\sim_G$ as desired.
\end{proof}

\begin{prop}\label{l:percoex}
Let $\sim$ be a lamination in $\smp$ with rotational gap $G$ which
co-exists with a quadratic invariant gap $U$.
Moreover, suppose $G\subset U$.
Then either $\sim=\sim_G$ or $M=M(U)$ is a leaf of $\sim$.
\end{prop}

\begin{proof}
By Lemma~\ref{l:spec-GvsL}, if $G$ is of type $D$ (whether $G$ is finite
or infinite) we have that $\sim=\sim_G$.  So we may assume that $G$
has a single orbit of edges.
Suppose that $\sim\ne\sim_G$, and $M$ is not a leaf of $\sim$.
Then $M$ lies in a periodic gap $H$ of $\sim$.  Let $M''$ be the chord of $V(U)$
with the same image as $M$.  Then by Lemma~\ref{l:spec-GvsL}, $M''$ is
also a chord which lies in $H$.  This immediately implies that $H$ is
a Fatou gap as laminations do not have finite critical gaps and Siegel gaps
have no periodic vertices.  We next want to establish that $H$ is not
quadratic. Otherwise, let $c$ be a critical chord joining an
endpoint of $M$ to an endpoint of $M''$.
Then $c_2=\psi_H(c)$ is a critical chord in $\uc$ and $\psi_H$-images of the
endpoints of $M$ are distinct points of the same period whose $\si_2$
orbits always stay on the same side of $c_2$, a situation which is
well-known to be impossible.

The result now follows easily. If $G$ is a finite gap, then $G$ has
a cycle $\F$ of Fatou gaps attached to its edges.  These must be
quadratic; otherwise $\sim=\sim_G$ (cf. the proof of Theorem
\ref{t:cormin-spec}). Then $H$ is also forced to be quadratic, a
contradiction with the first paragraph of the proof.  If $G$ is
Siegel, then one criticality is occupied by the critical edge of
$G$.  Once more, this forces $H$ to be quadratic which impossible by
the above.
\end{proof}

\bibliographystyle{amsalpha}

\end{document}